\documentclass{article}
\usepackage{graphicx}

\usepackage[fleqn]{amsmath}
\usepackage{amsthm}
\usepackage{amsmath}
\usepackage{amssymb}
\usepackage{amsbsy}
\usepackage{amsfonts}
\usepackage{mathrsfs}
\usepackage[all]{xy}
\usepackage{amstext}
\usepackage{amscd}
\usepackage[dvips]{epsfig}
\usepackage{psfrag}
\usepackage{enumerate}
\usepackage{flafter}
\allowdisplaybreaks

\theoremstyle{plain}
\newtheorem{thm}{Theorem}[section]
\newtheorem{prop}[thm]{Proposition}
\newtheorem{cor}[thm]{Corollary}
\newtheorem{lem}[thm]{Lemma}
\theoremstyle{definition}
\newtheorem{exa}[thm]{Example}

\newtheorem{rem}[thm]{Remark}
\newtheorem{defn}[thm]{Definition}

\def\Ker{\mathop{\mathrm{Ker}}\nolimits}
\def\Coker{\mathop{\mathrm{Coker}}\nolimits}

\def\F{\mathop{\mathbb{F}}\nolimits}

\newcommand{\tri}{{ \lhd}}

\newcommand{\lra}{\longrightarrow}
\newcommand{\ra}{\rightarrow}

\newcommand{\R}{{\Bbb R}}
\newcommand{\Z}{{\Bbb Z}}
\newcommand{\N}{{\Bbb N}}
\newcommand{\As}{{\rm As }}

\newcommand{\K}{{\mathcal{K}}}

\newcommand{\sh}{\mathcal{S}}
\newcommand{\C}{\mathbb{C}}

\newcommand{\pc}[2]{\mbox{$\begin{array}{c}
\includegraphics[scale=#2]{#1.eps}
\end{array}$}}
\begin{document}
\large
\begin{center}
{\bf\Large On the fundamental 3-classes of knot group representations}
\end{center}

\vskip 2.5pc
\begin{center}
{\Large Takefumi Nosaka}
\end{center}\vskip 1pc
\begin{abstract}\baselineskip=13pt
\noindent
We discuss the fundamental (relative) 3-classes of knots (or hyperbolic links),
and provide diagrammatic descriptions of the push-forwards with respect to every link-group representation.
The point is an observation of a bridge between the relative group homology and quandle homology from the viewpoints of
Inoue--Kabaya map \cite{IK}.
Furthermore, we give an algorithm to algebraically describe the fundamental 3-class of any hyperbolic knot.
\end{abstract}

\normalsize
\baselineskip=11pt
\begin{center}
\end{center}

\tableofcontents

\large
\baselineskip=16pt
\section{Introduction.}\label{s11}
In the study of an oriented compact 3-manifold $M$ with torus boundary,
the (relative) fundamental homology 3-class $ [M, \partial M] $ in $ H_3( M, \partial M ;\Z ) \cong \Z $ essentially has basic information.
To further analyze it quantitatively,
relative viewpoints are practice: 
To describe this, we suppose a pair of groups $K \subset G$ and a homomorphism $f : \pi_1 (M) \ra G$,
which sends every boundary element in $\pi_1 (M)$ to some element in $K $.
Then, for any relative group 3-cocycle $ \theta \in H^3(G,K;A)$ with local coefficients (see \eqref{k1} for the explicit definition),
we can consider the following pairing valued in the coinvariant $A_G= H_0(G;A )$:
\begin{equation}\label{ssa}\langle \theta , \ f_* [M, \ \partial M ] \rangle \in A_G= A/\{ a - g\cdot a \}_{a \in A , \ g\in G}.
\end{equation}
Thus setting appears in topics in low dimensional topology. For example, the volumes and the Chern-Simons invariants of hyperbolic manifolds can be
described as this pairing, where $G= SL_N(\mathbb{C})$ (see \cite{Neu,Zic,GTZ}).
Furthermore, this pairing includes the triple cup products of the form $ \theta=a\smile b \smile c$; see \cite{MS,Nos9}.
In addition, if $G$ is of finite order, the pairing is called {\it the Dijkgraaf-Witten invariant} \cite{DW}, as a toy model of TQFT.


However, the pairing defined in the general situation is often considered to be uncomputable.
Actually, we come up against difficulties:
First, it is troublesome to explicitly describe a (truncated) triangulation in $M $, which represents the 3-class $f_* [M, \partial M ] $.
Moreover, the 3-class $f_* [M, \partial M ] $ is not always unique, but depends on the choices of $2^{|\pi_0(\partial M)|}$ decorations, as mentioned in \cite[\S 5]{Zic}.
Next, the boundary condition is important; when dealing with the condition,
we mostly need long and verbose explanations, as in \cite{DW,Neu,Zic,Mor,Kab} (cf. Homotopy quantum field theory \cite{Tur}).
In addition, since the relative homology is defined from some projective resolution (see \S \ref{s11}),
it is essentially a critical problem to choose an appropriate resolution and to find a presentation of the 3-cocycle $ \theta$.
Further, even if we can succeed in doing so on $\theta$, 
such presentations are mostly quite intricate.

\

Nevertheless, this paper develops a diagrammatic computation of the pairing, according to geometric structures of links.
Precisely, let $L\subset S^3 $ be an oriented link in the 3-sphere, and
$M=E_L $ be the 3-manifold which is obtained from $S^3$ by removing an open tubular neighborhood of $L$, i.e., $ E_L=S^3 \setminus \nu L.$
When $L$ has ``malnormal property" as a broad class (i.e., hyperbolic links),
we succeed in giving a computation of the pairing; see Theorem \ref{mainthm1} (however, in cable cases, we need some conditions; see \S \ref{2734}, cf. cabling formula \cite{Ishi}).
As seen in \S \ref{23433}, the result is summarized to that,
if we know the presentation of $\theta $ and the JSJ decomposition of $L $, we can compute the pairing from a diagram. Here the point is that the construction needs no triangulation of $E_L$.

Let us roughly explain our approach to the theorem.
As seen in \cite{CKS,IK,Nos7,Nos3}, quandle theory \cite{Joy} and homology \cite{CKS} have advantages to some diagrammatic computation in knot theory.
Thus, inspired by the works \cite{IK,NM},
we will construct a bridge between the quandle and group homologies using chain maps,
in order to reduce the 3-class $[M, \partial M]$ to a quandle 3-class. 
However, as a technical reason appearing in the scissors congruence in \cite{Dup,Neu},
the bridge factors through Hochschild relative homology \cite{Hoc}, and
is formulated as a zigzag sequence; see \eqref{tukau2} in \S \ref{s11}.
As a solution, this paper points out (Theorem \ref{mainthm133})
that malnormal property of groups is a suitable condition for obtaining a quasi-inverse in the zigzag; see \S \ref{23433}
(Here, we are based on \cite{Sim,HW,NM} which studied malnormal property of knots).
To summarize, composing the chain maps gives the required computation of the pairing.

In applications, we obtain four advantages from the approaches as follows.
First, the above composite gives an algorithm to describe algebraically the fundamental 3-class $[M, \ \partial M ] $ for hyperbolic links; see \S \ref{238}.
Next, our results emphasize topological advantages of the quandle cocycle invariant \cite{CKS}.
Especially, for malnormal pairs ($G,K$), we will give a method to produce many quandle cocycles, and obtain a simple formulation of computing the pairing (see Theorem \ref{mainthm133}).
The third is a result of determining the third homology of the link quandle $Q_L$, where $L$ is a knot or a hyperbolic link.
This quandle $Q_L$ is analogous to the fundamental groups of $S^3 \setminus L$ (defined in \cite{Joy}), and plays a key role in the proof of the main theorems;
see Appendix \ref{2343136} for details.
The forth one is that our theorem is a generalization and application of the work \cite{IK}.
To be precise, while the paper \cite{IK} showed the same theorem for only $G=SL_2(\mathbb{C})$ and hyperbolic links,
our theorem points out the generalization applicable to groups $K \subset G$ with malnormality. 

This paper is organized as follows.
Section 2 states the theorems.
Section 3 introduces relative group homology,
and Section 4 reviews the quandle homology \cite{CKS} and Inoue--Kabaya chain map \cite{IK}.
Section 5 explains the algorithm to describe $[M, \ \partial M ] $, and gives an example from the figure eight knot.
Section 6 proves the main theorem, and Section 7 discusses cable knots.
Appendix A computes the third homology of the link quandles $Q_L$ for some links.


\section{Statements; the main results.}\label{s13}

This section states the main results.
For this, we fix terminology 
throughout this paper.

\vskip 0.49pc

\noindent
{\bf Conventional notation and assumption throughout this paper.}
\begin{itemize}
\item By a link we mean
a $C^{\infty}$-embedding of solid tori into the 3-sphere $S^3$ or into the solid torus $D^2 \times S^1$.
We suppose an orientation of $L$, and denote $\pi_0(L ) $ by $\# L \in \Z.$
\item For short, the fundamentals group $\pi_1(S^3 \setminus L)$ is abbreviated to $\pi_L$,
and the complement space $ S^3 \setminus L$ is often done to $E_L.$
\item Furthermore, fix a pherihedral group, $ \mathcal{P}_{\ell}$, with respect to $\ell \leq \# L $,
which is generated by a meridian-longitude pair $(\mathfrak{m}_{\ell},\mathfrak{l}_{\ell})$.
\item Moreover, we fix a group $G$ and subgroups $K_{\ell}$ with $ \ell \leq \# L$,
and suppose a homomorphism $ f : \pi_1( S^3 \setminus L ) \ra G$ such that $ f (\mathcal{P}_{\ell} ) \subset K_{\ell} $.
\end{itemize}

In this situation,
although it seems easy to define a pushforward of the 3-class $f_*( [E_L ,\partial E_L])$ in the relative group homology $H_3(G,K_1, \dots, K_{\# L} ;\Z)$,
it is known (see \cite[\S 5]{Zic} or \cite[\S 10]{NM}) that a canonical definition of such pushforwards depends
on the choice of ``$f$-decorations".
However, if $E_L$ is decomposed as a union of complete hyperbolic 3-manifolds, the 3-class $f_*( [E_L ,\partial E_L])$ is known to be well-defined (see \cite[\S 5]{Zic}).
Therefore, similarly to \eqref{ssa}, we can consider the pairing between this 3-class and a 3-cocycle of $G$ relative to $(K_{\ell})_{\ell \leq \# L} $.

\subsection{The first statement.}\label{s11131}
We will set up some terminology, and state Theorem \ref{mainthm1}.
Take the set of the arcs of $D$ and the set of the complementary regions of $D$, denoted by $\mathrm{Arc}_D $ and $\mathrm{Reg}_D $, respectively.
Given a map $ \phi: \mathrm{Reg}_D \times \mathrm{Arc}_D \times \mathrm{Arc}_D \ra A $,
let us consider a weight sum of the form
\begin{equation}
\label{1122} \Psi_{\phi } ( D) := \sum_{\tau}\epsilon_{\tau}\cdot \phi (x_{\tau },y_{\tau },z_{\tau }) \in A
\end{equation}
running over all the crossings $\tau $ of $D$, where
$x_{\tau }$, $y_{\tau }$, and $z_{\tau }$ are the region and the arcs shown in Figure \ref{fig.color43}, and
$\epsilon_{\tau} \in \{ \pm 1\}$ is the sign of $\tau$.
Then, the main statement is as follows:
\begin{figure}[htpb]
\begin{center}
\begin{picture}(100,80)
\put(-82,35){\Large $x_{\tau } $}
\put(-51,63){\Large $y_{\tau } $}
\put(-24,62){\Large $z_{\tau } $}
\put(-90,33){\pc{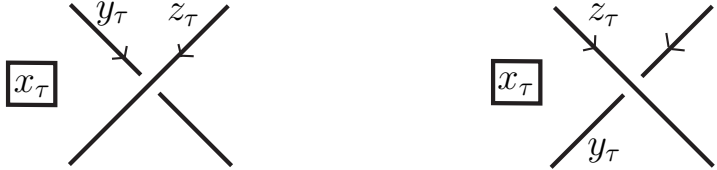}{0.40003285}}

\put(101,36){\Large $x_{\tau } $}
\put(135,62){\Large $z_{\tau } $}
\put(135,11){\Large $y_{\tau } $}
\end{picture}
\end{center}
\vskip -1.737pc
\caption{Positive and negative crossings with labeled regions and labeled arc.
\label{fig.color43}}
\end{figure}
\begin{thm}[See \S \ref{s33} for the proof.]\label{mainthm1}
Assume that $L$ is either a prime knot whish is not cable \footnote{A knot $K$ is {\it cable}, if there is a solid torus $V$ embedded in $S^3$ such that $V$ contains $K$ as the $(q)$-torus knot for some $p,q\in \Z.$
Furthermore, a knot $L$ is {\it prime}, if it cannot be written as the knot sum of two non-trivial knots.
Incidentally, the assumption in the theorems is inspired by Theorem \ref{tuu333} of \cite{HW} and the JSJ decomposition; see \S \ref{23433}.}
or a hyperbolic link.
Then, for any relative group 3-cocycle $ \theta \in H^3(G;K_1, \dots, K_{\# L};A)$,
there is a map
$$ \phi_{\theta }: \mathrm{Reg}_D \times \mathrm{Arc}_D \times \mathrm{Arc}_D \lra A $$
for which the following equality holds in the coinvariant $A_G$:
\begin{equation}
\label{11} \langle \ \theta , \ f_*[E_L, \ \partial E_L ] \rangle = \Psi_{\phi_{\theta } } ( D) \in A_G. \end{equation}
\end{thm}
In conclusion, the right hand side ensures a computation of the pairing, with describing no triangulation.
Here, it is important to express $ \phi_{\theta }$ concretely;
in Sections \ref{23433}--\ref{s33}, we give a concrete expression of $ \phi_{\theta }$ for such links.
However,
the presentation essentially depends on the link type of $L$, and is not always simple.
In fact, even for the figure eight knot $L$, the map $ \phi_{\theta }$ forms a sum of many terms; see Section \ref{238}.

Incidentally, Section \ref{2734} similarly discusses the cable cases, and concludes a similar statement (Theorem \ref{mainthm5}).
Here, we see that the statement essentially should be considered modulo some integers,
and that the pairing has no more information than the homology of cyclic groups.


\subsection{The second statement from malnormality and transfer.}\label{s111312}
In contrast, 
we will consider some conditions to get the map $ \phi_{\theta }$ and
the diagrammatic description in a concrete way.
Here, the subgroups $K_1,\dots, K_{\# L} \subset G$ are said to be {\it malnormal (in $G$)}, if they satisfy
\begin{enumerate}[($\star$)]
\item For any $(i,j) \in I^2$ and any $g \in G $ with $ g \not {\!\! \in } K_j$, the intersection $ g^{-1} K_i g \cap K_j $ equals $\{ 1_G \}.$
\end{enumerate}
The papers \cite{HWO,HW} give such examples, as in Gromov hyperbolic groups.
Furthermore, as in \cite{Agol,Wise}, the malnormality plays a key role in studying the virtual Haken conjecture and ``the Malnormal Special Quotient Theorem".

We will give a description of the pairing, by using quandle theory.
For this, we now review a quandle and colorings.
A {\it quandle} \cite{Joy} is a set, $X$, with a binary operation $\tri : X \times X \ra X$ such that
\begin{enumerate}[(I)]
\item The identity $a\tri a=a $ holds for any $a \in X. $
\item The map $ (\bullet \tri a ): \ X \ra X$ that sends $x $ to $x \tri a $ is bijective, for any $a \in X$.
\item The distributive identity $(a\tri b)\tri c=(a\tri c)\tri (b\tri c)$ holds for any $a,b,c \in X. $
\end{enumerate}
Let $X$ be a quandle. An $X$-{\it coloring} of $D$ is a map $\mathcal{C}: \mathrm{Arc}_D \to X$
such that $\mathcal{C}(\alpha_{\tau}) \lhd \mathcal{C}(\beta_{\tau}) = \mathcal{C}(\gamma_{\tau})$
at each crossings $\tau$ of $D$ illustrated as Figure \ref{fig.color}. Further, for $x_0 \in X$, a {\it shadow coloring} is a pair of an $X$-coloring $\mathcal{C} $ and a map $\lambda : \mathrm{Reg}_D \to X$ such that
the unbounded exterior region is assigned by $ x_0$ and that
if two regions $R $ and $R ' $ are separated by an arc $\delta $ as shown in the right of Figure \ref{fig.color}, then $\lambda(R ) \lhd \mathcal{C} (\delta )= \lambda (R ')$.
Here, notice that the assignment of every region is that of the unbounded region, by definition.
Thus, from arbitrary $x_0 \in X$ and $X$-coloring, we obtain uniquely a shadow coloring such that the unbounded region is labeled by $x_0$.

%

\vskip -1.419937pc
\begin{figure}[htpb]
\begin{center}
\begin{picture}(100,70)
\put(-122,46){\Large $\alpha_{\tau} $}
\put(-66,45){\Large $\beta_{\tau} $}
\put(-66,12){\Large $\gamma_{\tau} $}
\put(-82,1){\pc{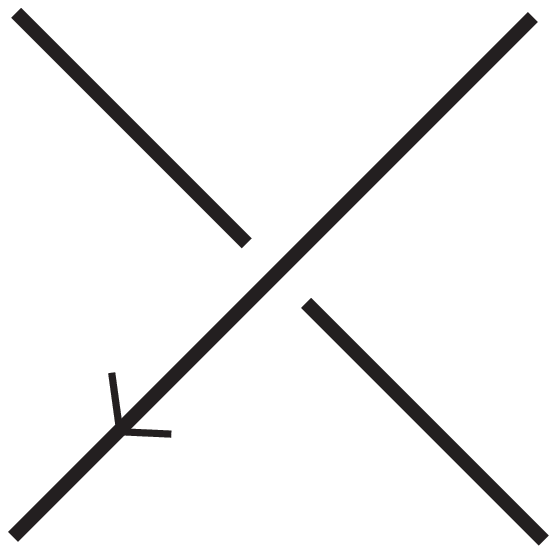}{0.2530174}}

\put(65,3){\pc{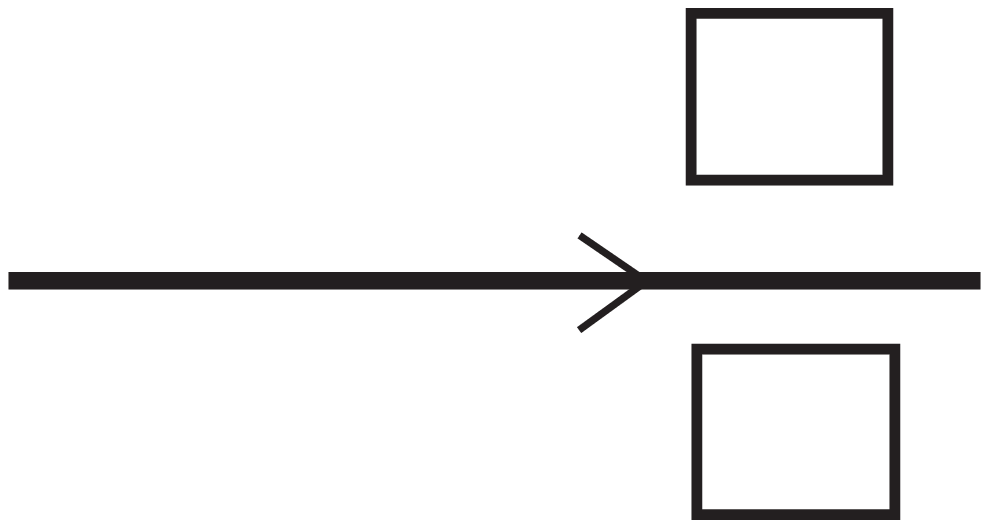}{0.266}}
\put(-50,16){$\mathcal{C}(\alpha_{\tau}) \lhd \mathcal{C}(\beta_{\tau}) = \mathcal{C}(\gamma_{\tau})$}

\put(88,43){\large $ \delta$}
\put(123,21){\large $R $}
\put(121,46){\large $ R ' $}
\put(164,35){\large $ \lambda(R) \lhd \mathcal{C} (\delta )= \lambda(R') $}
\end{picture}
\end{center}
\vskip -1.9937pc
\caption{The coloring conditions at each crossing $\tau$ and around each arcs. 
\label{fig.color}}
\end{figure}
\noindent

This paper mainly deals with the following class of quandles.
\begin{exa}[\cite{Joy}]\label{exaqSGH}
This example is due to Joyce \cite{Joy}.
Under the above settings $(f, G, K_{\ell})$, let $X$ be the union of the left quotients $ (K_{\ell }\backslash G)$, that is,
$$ X= \sqcup_{\ell=1}^{ \# L} (K_{\ell }\backslash G).$$
Let $k_{\ell}$ be $ f ( \mathfrak{m}_{\ell}) \in K_{\ell}. $
Assume that $ k_{\ell}$ is commutative with any elements of $K_{\ell}$. 
Then, the union $X$ is made into a quandle under the operation
\begin{equation}
\label{eq.gh}
[K_{\ell}x ]\lhd [K_{\ell'}y]= [K_{\ell} ^{-1}x y^{-1} k_{\ell'} y],
\end{equation}
for any $x, y \in G$. 
In what follows, we will write the triple $(G,\mathcal{K}) $ for this quandle.

Furthermore, as is known (see \cite[Appendix]{Nos7} for the details),
the homomorphism $f $ admits uniquely an $X$-coloring $ \mathcal{C}$ with $ \mathcal{C}(\mathfrak{m}_{\ell}) =k_{\ell}$
via Wirtinger presentation, where $ \mathcal{C}(\gamma)$ is defined to be $f(\gamma)^{-1 } k_{\ell} f(\gamma) $
if $\gamma$ lies in the $\ell$-the component of $L$.
Hence, by applying $k_1$ to $x_0$, we have the associated shadow coloring $\mathcal{S}:
\mathrm{Reg}_D \times \mathrm{Col}_D \to X $.
\end{exa}

Then, the main theorem in this subsection is stated as follows:
\begin{thm}\label{mainthm133}
Suppose that $ k_{\ell}$ commutates with any elements of $K_{\ell}$,
and that $L$ is a hyperbolic link or a prime knot which is neither a cable knot nor a torus knot, as in Theorem \ref{mainthm1}.
Furthermore, assume one of the following two:
(i) $(G,K_1 ,\dots, K_{\# L})$ is malnormal. \ (ii)
$K_1, \dots, K_{\# L}$ are of finite order and 
all the order $|K_i|$ is invertible in the coefficient group $A$.

Then, any relative group 3-cocycle of $(G,K_1 ,\dots, K_{\# L})$ is
represented as a map $ \theta: X^4 \ra A$ such that,
the fundamental 3-class $\langle \theta , \ f_* [E_L, \ \partial E_L ] \rangle$ is equal to
the sum
$$\sum_{\tau}\epsilon_{\tau} \bigl( \theta(k_1, a_{\tau }, b_{\tau },c_{\tau }) - \theta(k_1, a_{\tau } \lhd b_{\tau } , b_{\tau },c_{\tau }) -
\theta(k_1, a_{\tau } \lhd c_{\tau } , b_{\tau }\lhd c_{\tau },c_{\tau })+\theta(k_1, (a_{\tau } \lhd b_{\tau } ) \lhd c_{\tau } , b_{\tau }\lhd c_{\tau } ,c_{\tau }) \bigr) . $$
running over the all crossings $\tau $.
Here, for the assignment $(x_{\tau },\ y_{\tau },\ z_{\tau })$ around $\tau$ as in Figure \ref{fig.color43}, we define
$(a_{\tau },b_{\tau },c_{\tau }) \in X^3$ by setting $\bigl( \mathcal{S}( x_{\tau }) ,\ \mathcal{S}(y_{\tau }),\ \mathcal{S}(z_{\tau }) \bigr) $.

\end{thm}

\section{Relative group homology.}\label{s11}
First, we outline the proof of the theorems.
The key step is to introduce three chain groups, and
two chain maps (see \S\S \ref{s11}--\ref{s112}). 
These chain maps will be summarized to 
\begin{equation}\label{tukau2}
C_*^R (X ) \xrightarrow{\ \ \varphi_* \ \ } C_*^{\Delta } (X ; \Z ) \otimes_{\Z [\As(X)] } \Z \xleftarrow{\ \ \alpha \ \ } C_*^{\rm gr } (G ,\mathcal{K} ; \Z ) .
\end{equation}
Roughly speaking, the right hand side denominates relative groups cocycles (see \S \ref{s3s3w}),
and the left one can be diagrammatically described (see \S \ref{s112}). Thus, if we construct a chain map from the middle term to the right hand side, 
we can obtain diagrammatic computations as in the theorems.
However, as seen in \S \ref{s112}, the existence of such a chain map depends on some properties of knot type.
Thus, in the proof, we need careful verifications to deal with the chain maps.

To accomplish the outline in details
we first introduce relative group homology in the family version; see \S \ref{s3s3w}. After that,
we will give a key proposition \ref{tukau}, and define the chain map $\alpha$. 

Throughout this section, we fix a group $G$ and 
subgroups $K_1, \dots, K_m \subset G$ as above.
Furthermore, we denote the index set $\{1, \dots, m\} $ by $I$, and denote ($K_1, \dots, K_{m }$) by $\mathcal{K}$, for short. 
\subsection{Preliminaries; Two versions of group relative homology.}\label{s3s3w}
The relative group homology is usually defined from a group pair $K \subset G$; see, e.g., \cite[\S 3]{NM} or \cite{Zic}.
However, this paper generalizes the relative homology into the family version so as to deal with links. 

Consider the union of the left quotients, $\sqcup_{ i \in I } (K_i \backslash G)$,
and set up the subgroup of the form
\begin{equation}\label{k3331} C_n^{\rm red}(G, I):= \bigl\{ \ (a_1, \dots, a_m ) \in \Z[G^{n+1}]^m \ \bigr| \ \sum_{i \in I} a_i = 0\ \bigr\}.\end{equation}
Then, letting $n=0$, we canonically have a right $G$-module homomorphism
\[ P_I: C_0^{\rm red}(G, I) \lra \Z[ \sqcup_{ i \in I } (K_i \backslash G) ]. \]

We define {\it the relative group homology} of $(G,\mathcal{K})$ to be the torsion $\mathrm{Tor}_*^{\Z[G]}(\Coker (P_I),M) $,
where $M$ is a left $\Z[G]$-module.
Precisely, taking the augmentation map $\varepsilon: \Z [ \Coker (P_I)] \ra \Z$ with a choice of a projective resolution
$$ \mathcal{P}_*: \ \ \ \cdots \xrightarrow{\ \ \partial_{n+1} \ \ } \mathcal{P}_n \xrightarrow{\ \ \partial_{n} \ \ }
\cdots \stackrel{\partial_2}{\lra} \mathcal{P}_1 \stackrel{\partial_1}{\lra} \Coker (P_I) \xrightarrow{\ \ \varepsilon \ \ } \Z \ \ \ \ \ \ \ ({\rm exact}) ,$$
as right $\Z[G]$-modules, the relative homology is defined to be
$$ H_n (G, \K ;M ) :=H_n( \mathcal{P}_* \otimes_{\Z[G^n]} M, \partial_* ). $$
Dually, we can define the cohomology as $\mathrm{Ext}^*_{\Z[G]}(\Ker (\varepsilon ) ,M). $
For enough projectivity, we now cite an example of $ \mathcal{P}_* $.

\begin{exa}[{Mapping cone}]\label{rei1}
For any set $B$, we define the map $ \partial^{\Delta}_n :\Z[B^{n+1}] \ra \Z[B^{n}] $ by setting
\begin{equation}\label{g16} \partial_n^{\Delta}(x_0, \dots, x_n) = \sum_{i: \ 0 \leq i \leq n} (-1)^i (x_0, \dots, x_{i-1}, x_{i+1}, \dots, x_n )\in \Z[B^{n}] .
\end{equation}
Let $\iota_{\ell} : K_{\ell}\ra G$ be the inclusion.
Furthermore, according to \cite[\S 2]{Zic}), for $n>1,$ consider the following free $G$-module:
$$ C^{\rm gr}_n (G, K_{\ell}) := \bigl(( \Z[G^{n+1} ] \otimes_{\Z[G ]} \Z) \oplus ( \Z[( K_{\ell})^{n}] \otimes_{\Z[K_{\ell} ]} \Z )\bigr) \otimes_{\Z} \Z[G ] . $$
Furthermore, when $n=1,$ we define $ C^{\rm gr}_1 (G, \mathcal{K})=C_1(G)=\Z[G^2]$.
Then, we can easily see that the following assignments define a differential on these modules:
$\partial_1:= \partial_1^{\Delta}$ and
when $n>1,$
$$ \partial_n \bigl( \vec{g}, \ \vec{k}_\ell \bigr):= \bigl( \partial^{\Delta}_n( \vec{g}) + (-1)^n \iota_{\ell} ( \vec{k}_\ell ) , \
\partial^{\Delta}_{n-1}(\vec{k_\ell} ) \bigr)$$
for any $(\vec{g}, \ \vec{k}_\ell ) \in G^{n+1} \times K_\ell^{n} . $
From the definition in \eqref{k3331}, it is sensible to consider a canonical inclusion
\begin{equation}\label{gigi} P_n^{(I)} : C_n^{\rm red}(G, I) \lra \bigoplus_{j: 1 \leq j \leq n } C^{\rm gr}_n (G, K_j ).
\end{equation}
Then, we define the pair ($ C^{\rm gr}_* (G, \mathcal{K}) , \partial_* $) to be the cokernel of $P_n^{(I)} $.
Thus, we define $H_n^{\rm gr} ( G , \mathcal{K} ; M)$ to be the homology of the complex $( C^{\rm gr}_* (G, \mathcal{K})\otimes_{\Z[G]} M , \partial_*) $.

\begin{prop}\label{tuk32au}
Then, the pair ($ C^{\rm gr}_* (G, \mathcal{K}) , \partial_* $) gives a free resolution of $ \Coker (P_I) $.
\end{prop}
\begin{proof} Since the statement with $|I|= 0, 1 $ is known (see \cite[Theorem 2.1]{Zic}), 
we may assume $|I|>1 $.
Then, we have two sequences with commutativity:
$${\normalsize
\xymatrix{
\cdots \ar[r] & C_n^{\rm red}(G, I) \ar@{_{(}->}[d]^{P_n^{(I)}} \ar[r]^{(\partial_n)^m} & C_{n-1}^{\rm red}(G, I) \ar@{_{(}->}[d]^{P_{n-1}^{(I)}} \ar[r]^{ \ \ \ (\partial_{n-1})^m} & \cdots \ar[r] & C_0^{\rm red}(G,\emptyset, I) \ar[r]^{ \ \ \ \ \ \ \ \ \ \ \varepsilon}
\ar@{_{(}->}[d] & \Z & (\mathrm{exact}) \\
\cdots \ar[r]& \bigoplus C^{\rm gr}_n (G, K_i ) \ar[r] & \bigoplus C^{\rm gr}_{n-1} (G, K_j ) \ar[r] & \cdots \ar[r] & \bigoplus \mathrm{Coker}(P_{\{ j\}}) \ar[r]^{ \ \ \ \ \ \ \ \ \ \ \varepsilon} & \Z & (\mathrm{exact})
}}
$$
Here, the exactness is ensured by the cases with $|I|\leq 1$.
The cokernel of the vertical map is exactly the complex ($ C^{\rm gr}_* (G, \mathcal{K}) , \partial_* $) from definitions.
Hence, by the five lemma, the cokernel gives a free resolution of $\Coker (P_I) $.
\end{proof}
Then, a standard discussion of mapping cones deduces the long exact sequence with $n \geq 2$:
\begin{equation}\label{g666} \cdots \ra H_{n+1}^{\rm gr} ( G, \mathcal{K} ; M) \stackrel{\delta_*}{\lra} \oplus_j H_{n}^{\rm gr} ( K_j ; M )
\xrightarrow{ \ \oplus (\iota_j)_*\ } H_n^{\rm gr} ( G; M) \lra H_{n}^{\rm gr} ( G, \mathcal{K} ; M) \ra \cdots .
\end{equation}

\end{exa}

\begin{rem}\label{tuk392}
Here, we mention a topological meaning of 
the relative homology $H_n( G , \mathcal{K} ; M ) $.
Let $ K(K_j,1 )$ and $ K(G,1 )$ be the Eilenberg-MacLane spaces of $K_j$ and of $ G$, respectively.
Let $ (\iota_j)_* : K(K_j,1 ) \ra K( G ,1 ) $ be the map induced from the inclusions.
Then, similarly to \cite[\S 2 and \S 5]{Zic}, we can see that $H_n( G , \mathcal{K} ; M ) $ is isomorphic to the homology of the mapping cone of $ \sqcup_j \iota_j : \sqcup_j K(K_j,1 )\ra K(G,1 )$
with local coefficients $M$.
\end{rem}

\begin{exa}\label{rei122}
As an example, let us describe 3-cocycles in the non-homogenous cochain, as $\Coker (P_I) $.
Specifically, a 3-cocycle of $G$ relative to $\mathcal{K} $ is represented by $m$ maps $ \theta_{\ell} : G^3 \ra M$ and $ \eta_{\ell} :(K_{\ell})^2 \ra M$
satisfying
the two equations
\begin{equation}\label{k1} g_1 \cdot \theta_{\ell} (g_2,g_3,g_4 ) - \theta_{\ell} (g_1 g_2,g_3,g_4 ) +\theta_{\ell} (g_1,g_2g_3,g_4 )- \theta_{\ell} (g_1, g_2,g_3g_4 )+ \theta_{\ell} (g_1, g_2,g_3 )=0,\end{equation}
\begin{equation}\label{k2} \theta_{\ell} (k_1,k_2,k_3 )= k_1\cdot \eta_{\ell} ( k_2,k_3 ) -\eta_{\ell} (k_1 k_2,k_3 )+\eta_{\ell} (k_1, k_2k_3 )-\eta_{\ell} (k_1 ,k_2 ),
\end{equation}
for any $g_i \in G$ and $ k_i \in K_{\ell}$, and $1 \leq \ell \leq m$.

\end{exa}

In another way, we will introduce the relative homology of $ \sqcup_{ i \in I } (K_i \backslash G)$, which is originally defined by Hochschild \cite{Hoc}.
Let $Y$ be the union $\sqcup_{ i \in I } (K_i \backslash G) $, and let $C_n^{\rm pre }(Y)$ be the free $\Z$-module generated by $(n+1)$-tuples $(y_0, y_1, \dots, y_n) \in Y^{n+1} $.
Consider the differential homomorphism defined by $\partial^{\Delta}_*$ as above.
As is basically known (see \cite{NM,Bro,Zic}), the chain complex $ (C_*^{\Delta}(Y) , \ \partial_{*}^{\Delta }) $ is acyclic.
From the natural action $Y \curvearrowleft G$,
let us equip $C_n^{\rm pre }(Y)$ with the diagonal action.
Furthermore, as a parallel to \eqref{k3331}, from the definition
of $C_n^{\rm red}(G, I)$,
we can similarly consider a $\Z[G] $-homomorphism $Q_n: C_n^{\rm red}(G, I) \ra C_n^{\rm pre }(Y)$.

\begin{defn}\label{de3553}
We define the chain complex $ (C_*^{\Delta}(Y) , \ \partial_{*}^{\Delta }) $ to be
the cokernel $\Coker(P_n) $, which is diagonally acted on by $G$.

Furthermore, $H_*^{\Delta}(Y;M)$ denotes the homology of the quotient complex $\Coker(P_n)\otimes_{\Z[G]} M$.
Namely, $H_*^{\Delta}(Y;M) =H_*( \Coker(Q_n)\otimes_{\Z[G]} M ) $.
\end{defn}
This chain complex $ (\Coker(Q_*) , \ \partial_{*}^{\Delta }) $ is acyclic and is not always projective, even if $|I|=1$. 
However, the projectivity of $ \mathcal{P}_* $ admits, uniquely up to homotopy, a chain $\Z[G]$-map
\begin{equation}\label{77777}\alpha: ( \mathcal{P}_*, \partial_*) \lra ( C_*^{\Delta}(Y), \partial_*^{\Delta}).
\end{equation}
\begin{exa}\label{tukau5} When $ \mathcal{P}_* $ is the complex $ C^{\rm gr}_* (G, \mathcal{K}) $ in Example \ref{rei1}, we will give an example of $\alpha $.
Consider the normalized complex of $ \Coker(Q_n)$ subject to the submodule
$$\Z \langle (y_0,\dots, y_n) \in Y^n \ | \ y_{i} =y_{i+1} \textrm{ for some } i \ \rangle,$$
and denote it by $ C_*^{\rm Nor}(Y)$. As usual in normalization, we note $H_*^{\Delta}(Y;M) \cong H_*^{\rm Nor}(Y;M) $. Then, for $1 \leq j \leq m$, consider
the correspondence
$$ \alpha_j^{\rm pre }:G^{n+1} \times (K_j)^n \lra C_n^{\rm pre}( K_j\backslash G); \ (g_0,\dots, g_n,k_0,\dots, k_{n-1}) \longmapsto (K_jg_0 , K_j g_1 \dots, K_j g_n). $$
Here let us regard $ \mathcal{P}_*= C^{\rm gr}_*(G,\mathcal{K} ) $ as the cokernel $ \mathrm{Coker} (P_n^{(I)})$; see \eqref{gigi}. Subject to the image of $ C_n^{\rm red}(G, I) $, the direct sum of $ \alpha_j^{\rm pre } $ yields
a chain map $\alpha : C^{\rm gr}_* (G, \mathcal{K}) \ra C_*^{\rm Nor}(Y) $.
\end{exa}

\subsection{A key proposition from malnormality, and some examples.}\label{s3112}
Whereas this $\alpha$ is not always a quasi-isomorphism (see \cite[\S 3.2]{NM} for counter-examples), we give a criterion which is a key in this paper.
\begin{prop}[{A modification of \cite[Proposition 3.23]{NM}}]\label{tukau}
The set $Y=\sqcup_{ i \leq \# L} (K_i \backslash G)$ is assumed to be 
of infinite order. Furthermore, the subgroups $K_1, \dots, K_{\# L} \subset G$ are malnormal.

Then, 
the chain map $\alpha $ induces an isomorphism $H_*^{\Delta}(Y;M ) \cong H_* (G, \mathcal{K} ;M ) $ for any coefficient $M$.
\end{prop}
\begin{proof}
The proof is essentially due to \cite{NM}. Consider the submodule 
\begin{equation}\label{777}
C_n^{\neq }(Y):= \Z \{ \ [(y_0, \dots,y_n)]\in \Coker(Q_n) \ \ | \ \ \mathrm{If \ } s \neq t \mathrm{, \ then \ } y_s \neq y_t . \ \}.
\end{equation}
Since $Y$ is of infinite order, this $C_n^{\neq }(Y) $ is an acyclic subcomplex of $\Coker(Q_n) $,
and the injection is quasi-isomorphic; see \cite[Proposition 3.20]{NM} for the details.
Furthermore, we can easily check that, if $ \sigma \cdot g = \sigma $ with $g \in G$ and $\sigma:= (y_0, \dots, y_n) \in C_n^{\neq }(Y)$,
the malnormal assumption implies $g=1_G \in G$.
That is, the action is free; therefore, $C_n^{\neq }(Y) $ is a free $\Z[G]$-module.
Since the above $\alpha $ factors through $C_n^{\neq }(Y) $, we have the conclusion.
\end{proof}

We end this section by giving three examples satisfying the assumptions.
\begin{exa}[hyperbolic 3-manifolds]\label{tuu}
Let $N$ be a compact hyperbolic 3-manifold with torus boundary $\partial N =T_1 \sqcup \cdots \sqcup T_m$.
Apply $G$ to $\pi_1(N)$ and $K_i $ to $ \pi_1(T_i)$ with a choice of base point.
As is well-known as ``algebraic atoroidality" in hyperbolic geometry (see \cite{AFW}),
the boundary group $ \pi_1(T_i)$ injects $\pi_1(N)$, and
the malnormal condition holds.
Since $N $ is also a $K(G,1)$-space by hyperbolicity, we thus have the isomorphisms
$$H_*^{\Delta}(Y;\Z ) \cong H_*^{\rm gr}( G, \mathcal{K};\Z ) \cong H_*(N, \partial N;\Z ) . $$
\end{exa}

\begin{exa}[Knots]\label{knotmal}
Furthermore, given a non-trivial knot $ L$ in the 3-sphere $S^3$, we replace $G$ by $\pi_1(S^3 \setminus L)$ and $K_1$ by a peripheral subgroup $\pi_1( \partial (S^3 \setminus L))\cong \Z^2 $, which is
generated by a meridian-longitude pair $(\mathfrak{m}, \ \mathfrak{l})$.
By the loop theorem of 3-manifolds, $K_1$ injects $ G$.
Furthermore, 
$S^3 \setminus L$ is basically known to be a $K(G,1)$-space.

Moreover, we mention a theorem to detect the malnormality of the knot group.

\begin{thm}[\cite{Sim,HW}]\label{tuu333}
Let $ K_1\subset G$ be as above. The pair $(K_1,G)$ is malnormal if and only if the knot $L$ is none of the following three cases:
torus knots, cable knots, and composite knots.

In particular, in the case, 
the isomorphism $\alpha : H_*^{\Delta}(Y;\Z ) \cong H_*(E_L, \partial E_L ;\Z ) $ holds.
\end{thm}
\end{exa}


\begin{exa}[Link quandles]\label{rei3434}
More generally, let us consider a link $ L \subset S^3$ and the link group $G= \pi_L =\pi_1(S^3 \setminus L)$.
Let $ K_\ell $ with $ 1 \leq \ell \leq \pi_0(L) $ be the abelian subgroup generated by a meridian-longitude pair
$(\mathfrak{m}_\ell ,\mathfrak{l}_\ell )$ with respect to the $\ell $-th link component, that is, $K_\ell $ is a peripheral group generated by $(\mathfrak{m}_\ell ,\mathfrak{l}_\ell )$.
We denote $\sqcup_{\ell } K_\ell $ by $\partial \pi_L $ hereafter.

However, there are many links satisfying non-malnormality on $ (\pi_L , \partial \pi_L) $, as in the Hopf link.
More generally, malnormality for non-splittable links is completely characterized in \cite[Corollary 4]{HW}.

Incidentally, the union $\sqcup_{j} ( K_j \backslash \pi_L ) $ with the binary operation (\ref{eq.gh}) is called {\it the link quandle} \cite{Joy}.
We denote the link quandle by $Q_L$, since we later use it in many times.
\end{exa}

\section{Review; quandle homology and Inoue--Kabaya map.}\label{s112}
Next, regarding the middle term in the zigzag sequence \eqref{tukau2},
this section reviews the quandle homology \cite{CJKLS} and Inoue--Kabaya chain map \cite{IK}.
As seen in \cite{CKS,IK,Nos7,Nos3}, quandle theory is useful for reducing some 3-dimensional discussions to diagrammatic objects.

We briefly explain the rack and quandle (co)homology groups \cite{CJKLS,CKS}.
Let $X$ be a quandle, and $C_n^R(X)$ be the free right $\Z$-module generated by $ X^n$. Namely, $ C_n^R(X):= \Z [X^n]. $
Define a boundary $\partial^R_n : C_n^R(X) \rightarrow C_{n-1}^R(X )$ by
$$ \partial^R_n ( x_1, \dots,x_n)= \sum_{1\leq i \leq n} (-1)^i\bigl( ( x_1\tri x_i,\dots,x_{i-1}\tri x_i,x_{i+1},\dots,x_n)
-(x_1, \dots,x_{i-1},x_{i+1},\dots,x_n) \bigr).$$
Since $\partial_{n-1}^R \circ \partial_n^R =0$ as usual, we can define the homology $H^R_n(X)$ and call it {\it the rack homology}.
Furthermore, let $C^D_n (X) $ be the submodule of $C^R_n (X)$ generated by $n$-tuples $(x_1, \dots,x_n)$
with $x_i = x_{i+1}$ for some $ i \in \{1, \dots, n-1\}$.
One can easily see that this $C^D_n (X) $ is a subcomplex of $ C^R_{n} (X). $
Then, the {\it quandle homology}, $H^Q_n (X ) $, is defined to be the homology of the quotient complex $C^R_n (X ) /C^D_n (X)$.
In general, it is not easy to compute these homology groups.

In addition, we will review the Inoue--Kabaya map whose codomain is the Hochschild complex in Definition \ref{de3553}.
For this, we need some notation.
A map $f: X \ra X'$ between quandles is {\it a quandle homomorphism}, if $f(a\tri b)=f(a)\tri f(b)$ for any $a,b \in X$.
Furthermore, given a quandle $X$, we set up the abstract group, $\As (X) $, with presentation
$$\mathrm{As}(X) := \langle \ e_x \ (x \in X )\ | \ e_{x\tri y}^{-1} \cdot e_y^{-1} \cdot e_x \cdot e_y \ (x,y \in X)\ \ \rangle.$$
We call $\mathrm{As}(X)$ {\it the associated group}.
Further, $\mathrm{As}(X)$ has a right action on $X$ defined by $x \cdot e_y:= x \lhd y $, where $x,y \in X .$
Let $ O(X)$ be the orbit set of $X \curvearrowleft \As (X)$.
With respect to $i \in O(X)$, we fix $x_i \in X$ in the orbit.
As in Example \ref{exaqSGH}, denoting $\mathrm{Stab}(x_i) \subset \As (X)$ by $K_i$,
we can consider the setting
$$X=Y= \sqcup_{ i \in O(X)} (K_i \backslash G) \ \ \mathrm{with} \ \ G= \As (X) .$$
Furthermore, we set up the following set consisting of some maps:
$$ I_n := \bigl\{ \ \iota : \{2,3, \dots, n\} \lra \{ 0,1\} \ \bigr\},$$
which is of order $2^{n-1} $.
Moreover, given a tuple $(x_1 ,\dots, x_n) \in X^n $ and each $\iota \in I_n$, we define
$x(\iota, i) \in X$ by the formula
$$ x(\iota, i):= x_i \cdot (e_{x_{i+1}}^{\iota(i+1)} \cdots e_{x_{n}}^{\iota(n)}).$$

Then, with a choice of an element $p \in X$, we define a homomorphism
$$ \varphi_n: C_n^R(X ) \lra C_n^{\rm pre}(X) \otimes_{\Z [\As(X)] } \Z $$
by setting
$$ \varphi_n(x_1,\dots,x_n) := \sum_{\iota \in I_n} (-1)^{ \iota(2)+\iota(3)+\cdots +\iota (n)}\bigl( x(\iota,1),\dots, x(\iota, n) ,p \bigr). $$
Here are the descriptions of $\varphi_*$ of lower degree:
\begin{eqnarray*}
\varphi_2(a,b)&=&(a,b ,p)-(a\lhd b ,b ,p), \\
\varphi_3(a,b,c)&=&(a,b,c ,p)-(a\lhd b ,b,c ,p)-(a\lhd c ,b\lhd c,c ,p)+ ((a\lhd b) \lhd c ,b\lhd c,c ,p) .
\end{eqnarray*}
Then, it is shown \cite[\S 4]{IK} that this $\varphi_n$ is a chain map, i.e., $ \partial_n^{\Delta}\circ \varphi_n =\varphi_{n-1} \circ \partial_{n}^{R}$,
and that 
if $n \leq 3$, the image of the subcomplex $ C^D_n (X) $ is nullhomotopic.
Hence, the map $ \varphi_3$ with $n=3 $ induces a homomorphism $$ (\varphi_3)_*:H_3^Q(X ) \lra H_3^{\Delta}(X;\Z) .$$
We refer the reader to several studies on the chain map; see \cite{IK,Kab,Nos3,Nos2,Nos7}.

Next, we review the quandle cocycle invariant.
Given a shadow coloring $\sh$ of a link diagram $D$,
{\it the fundamental 3-class of $\sh$}, denoted by [$\sh$], is defined to be the sum
$$[\sh] : =\sum_{\tau}\epsilon_{\tau} \bigl(\lambda (x_{\tau }) ,\ \mathcal{C} (y_{\tau }),\ \mathcal{C} (z_{\tau }) \bigr) \in C_3^Q( X) $$
running over all the crossings $\tau $, where
the triple $(x_{\tau },\ y_{\tau },\ z_{\tau })$ are the three assignments around $\tau$ illustrated in Figure \ref{fig.color43}, and
$\epsilon_{\tau} \in \{ \pm 1\}$ is the sign of $\tau$.
Then, we can easily see that $[\sh]$ is a quandle 3-cycle in $C_3^Q( X) $; see \cite{CKS}.
If we have a quandle 3-cocycle $\phi: X^3 \ra A$,
the pairing
$\langle \phi, [\sh]\rangle \in A $ is called {\it the quandle cocycle invariant of $\sh$}.
Here, a map $\phi: X^3 \ra A$ is a quandle 3-cocycle, if the followings hold by definition:
$$ \phi(x,z,w)- \phi(x \lhd y ,z,w)- \phi(x,y,w)+ \phi(x \lhd z ,y \lhd z,w)= \phi(x \lhd w, y \lhd w,z \lhd w) - \phi(x,y,z),$$
$$ \phi(x,x,y)= \phi( y ,z ,z)=1 , \ \ \ \ \ \mathrm{for \ any \ \ } x,y,z,w \in X. $$
For calculating the invariant $\langle \phi, [\sh]\rangle $, it is important to find explicit formulas of quandle 3-cocycles $\phi $, although it is difficult in general.

\begin{exa}\label{r11}
Let $X$ be the link quandle $Q_L$ of a link; see Example \ref{rei3434}.
Consider the identity $\pi_L \ra \pi_L$, which induces $\mathrm{id}_{Q_L}: Q_L \ra Q_L $ . Then, we obtain from Example \ref{exaqSGH} the $Q_L$-coloring $\mathcal{S}_{\mathrm{id}_{Q_L}}$,
together with the associated 3-class $[\mathcal{S}_{\mathrm{id}_{Q_L}} ]$. This homology 3-class plays a key role later.
\end{exa}
\begin{exa}\label{rei322}
In the hyperbolic case, Inoue and Kabaya obtained a 3-cocycle from the chain map $\varphi_*$, with a relation to the Chern-Simons invariant.
Let $X$ be the quotient set $\mathbb{C}^2 \setminus \{ (0,0)\}/ \sim $
subject to the relation $ (a,b)\sim (-a,-b)$.
Equip $X$ with a quandle operation
\[\left(
\begin{array}{cc}
a & b
\end{array}
\right)
\lhd \left(
\begin{array}{cc}
c & d
\end{array}
\right)= \left(
\begin{array}{cc}
a & b
\end{array}
\right) \left(
\begin{array}{cc}
1+ cd & d^2 \\
-c^2 & 1-cd
\end{array}
\right). \]
One can easily verify that $ X$ is isomorphic to
the triple $(G, K,z_0) $ as in Example \ref{exaqSGH},
where $G$ is $PSL_2 (\mathbb{C})$
and $K$ is the unipotent subgroup of the form
${\small \Bigl\{ \left(
\begin{array}{cc}
1 & a \\
0 & 1
\end{array}
\right) \Bigr|\ a \in \mathbb{C} \Bigr\}} $, and $z_0={\small \left(
\begin{array}{cc}
1 & 1 \\
0 & 1
\end{array}
\right)} $.

Although this $(G,K )$ is not malnormal,
the paper \cite[\S 4]{NM} showed that 
the chain map $\alpha $ in \eqref{77777} is a quasi-isomorphism, which ensures a quasi-inverse $\beta$.
Furthermore, Neumann \cite{Neu} and Zickert \cite{Zic} described
the Chern-Simons 3-class as a relative group 3-cocycle
\begin{equation}\label{ki2n} \mathrm{CS} \in C^3_{\rm gr } ( PSL_2( \mathbb{C}), K ;\ \mathbb{C}/\pi^2 \Z),
\end{equation}
together with a cocycle presentation (see \cite{Zic} for the detail).
As a consequence, 
we concretely get a quandle 3-cocycle $ \varphi^* \circ \beta^*( \mathrm{CS})$.

Further, we consider a hyperbolic link $L$ and explain \eqref{kaa} below.
From the viewpoint of Example \ref{exaqSGH}, the associated holonomy representation $\rho : \pi_L \ra PSL_2(\C)$ is regarded as a shadow $X$-coloring $ \sh_{\rho}$.
Then, Inoue and Kabaya \cite[Theorem 7.3]{IK} showed the equality
\begin{equation}\label{kaa} \langle \rho^* ( \mathrm{CS}), \ [E_L, \ \partial E_L ] \rangle = \langle ( \beta \circ \varphi_3 )^* ( \mathrm{CS}) ,\ [\sh_{\rho}]\rangle \in \C /\pi^2 \Z . \end{equation}
Notice that, the right hand side is a quandle cocycle invariant, by definition.
As a result, we can compute the Chern-Simons invariant without triangulation; see \cite{IK} for examples.
\end{exa}
\section{Algebraic representation of the fundamental homology 3-class. }\label{23433}
In this section, we give a method to algebraically represent the fundamental 3-class $[E_L, \partial E_L]$,
where $L$ is either a hyperbolic link or a prime non-cable knot.

To describe this, the following plays a key viewpoint (see \S \ref{s33} for the proof).
\begin{thm}\label{clA1ee22}
Assume that a prime knot $L$ is neither a cable knot nor a torus knot, as in Theorem \ref{tuu333}.
Then, the Inoue--Kabaya chain map $\varphi_3 $ induces an isomorphism $ H_3^Q(Q_L) \ra H_3^{\Delta }(Q_L;\Z) $ as $\Z.$
\end{thm}
Theorem \ref{3331} implies that $ H_3^Q(Q_L)$ is generated by some fundamental 3-class [$\mathcal{S}_{\mathrm{id}_{Q_L}}$].
Hence, if we can explicitly formulate a quasi-inverse $\beta: C_*^{\Delta } (Q_L )_{ \pi_L } \ra
C_*^{\rm gr } (\pi_L ,\partial \pi_L ; \Z ) $ in \eqref{aa},
then we obtain an algebraic presentation of the fundamental 3-class $[E_L, \partial E_L]$.

We will explain the reason why we focus on only hyperbolic links in \S \ref{238}.
The key is the JSJ-decomposition of knots (or the geometrization theorem).
Precisely, as seen in \cite[Theorem 4.18]{Bud} or \cite{AFW,HW}, there exist open sets $ V_1 \subset V_2 \subset \cdots \subset V_n \subset S^3$ satisfying the followings:
\begin{enumerate}
\item The set $V_i$ for any $i$ is an open solid torus in $S^3$, and $V_i$ contains the knot $L$.
\item For any $ i\in \Z_{\geq 0}$, the difference $ V_{i} - \overline{V}_{i-1} $ is homeomorphic to either a composite knot or a hyperbolic knot or an $(n_i,m_i)$-torus knot in the solid torus for some $(n_i,m_i) \in \mathbb{Z}^2$.
Here we denote the knot $L$ by $V_0$.
\end{enumerate}
As is known, the decomposition is unique in some sense.
Here, remark (see \cite[Corollary 4.19]{Bud}) that $ L$ is a cable knot if and only if a difference $ V_{1} -\overline{V_0}$ is an $(n,m)$-torus knot in the solid torus; see Figure \ref{fig3lo3}.

Following the JSJ-decomposition, let us further examine the pairing \eqref{ssa}.
Denote the inclusion
$ V_i - V_{i-1} \subset S^3 -L $ by $\iota_i$, and the torus-boundary $\overline{ V_{i-1}} \cap \overline{ V_{i}}$ by $B_i$. 
Then, given $f : \pi_L \ra G$ and $\theta $ as before, we have $ f \circ (\iota_i)_*: \pi_1(V_i - V_{i-1} ) \ra G$.
Let $K_i \subset G$ be the image of $ \pi_1(B_i)$ via $ f \circ (\iota_i)_*$, where we appropriately choose a base point.
Then, we can regard the pullback $\iota_i^* \circ f^* (\theta)$ as a relative group 3-cocycle of $ (G,K_i,K_{i+1})$.
Therefore, the excision axiom on $\iota_i $'s ensures the equality
\begin{equation}
\label{13882} \langle f^* (\theta) , \ [E_L, \partial E_L] \ \rangle = \sum_{i \ : 1 \leq i \leq n } \langle \iota_i^* \circ f^* (\theta) , \ [ \overline{V_{i}} - V_{i-1} , \partial ( \overline{V_{i}} - V_{i-1} )] \ \rangle .
\end{equation}
To conclude, 
it is reasonable to deal with the fundamental 3-classes piecewise, according to the JSJ-decompositions of knots.

\subsection{The fundamental relative 3-class of hyperbolic links. }\label{238}
This subsection gives an explicit algorithm for describing the fundamental relative 3-class of hyperbolic links.
Here, the description is done in truncated terms (Theorem \ref{class}).

We begin by reviewing the truncated complex, which is defined by Zickert \cite[\S 3]{Zic}.
Fix a group $G$, and subgroups $K_1, \dots, K_m$.
For $n \geq 1$, consider the free abelian group $\Z[G^{n^2+ n }]$, and denote the $(ij)$-th generator $g \in G$ by $g_{ij}$ with $ i\neq j$.
Define $\overline{C}_n(G, \mathcal{K})$ by the submodule of $\Z[G^{n^2+ n }]$
which is generated by $g_{ij}$
satisfying 
\begin{itemize}
\item for any $i \in \{ 0,\dots, n\}$, there exists $m_{i} \in \{ 1,\dots, m\} $ such that the $n$ elements $g_{i0}, \dots, \check{g_{ii}}, \dots,g_{in}$ subject to $K_{m_i}$ are
equal in the coset $ K_{m_i} \backslash G. $
\end{itemize}
Then, right multiplication endows $\overline{C}_n(G, \mathcal{K})$ with a $\Z[G]$-module structure, and the usual simplicial boundary map gives rise to a boundary map $\partial_*$ on $\overline{C}_n(G, \mathcal{K})$.
The complex ($\overline{C}_*(G, \mathcal{K}), \partial_*$) is called {\it the truncated complex} of $(G,\mathcal{K})$.
As was similarly shown \cite[Remark 3.2 and Proposition 3.7]{Zic}, we can easily verify that this complex is a free resolution of
$\mathrm{Coker}(P_I)$. 

\begin{figure}[htpb]
\begin{center}
\begin{picture}(100,100)
\put(60,43){\pc{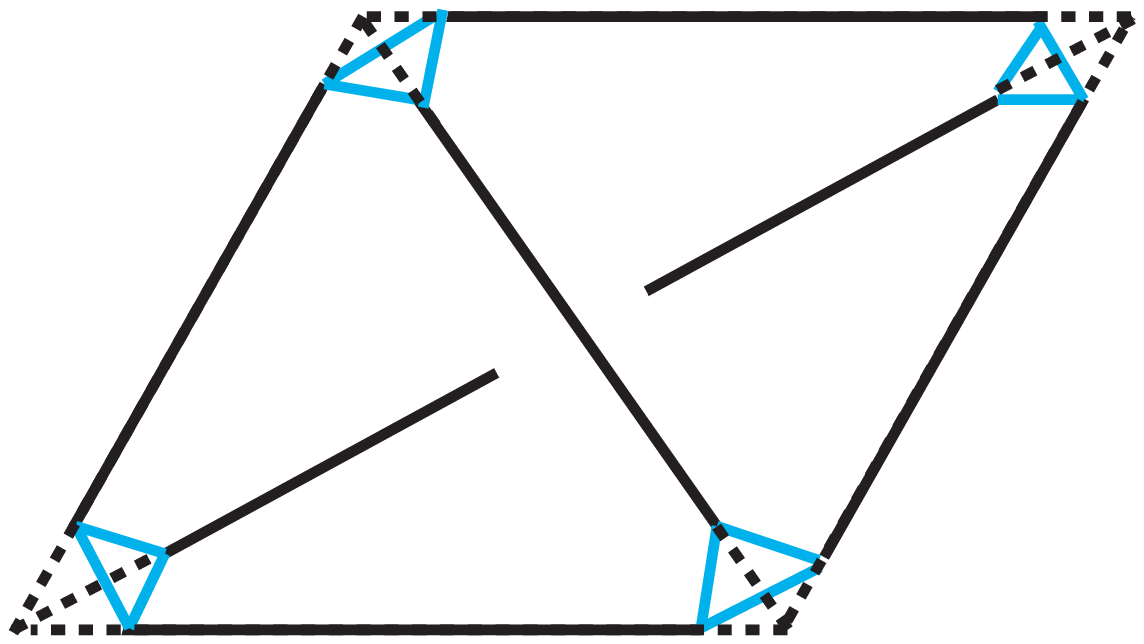}{0.41}}
\put(-130,43){\pc{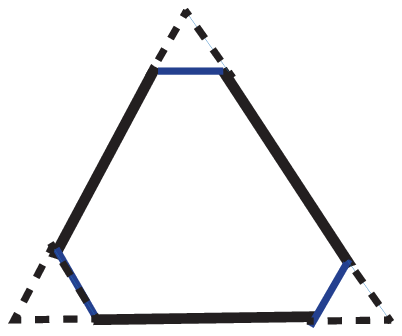}{0.65}}
\put(-105,10){\large $g_{01} $}
\put(-128,32){\large $g_{02} $}
\put(-67,10){\large $g_{10} $}
\put(-52,32){\large $g_{12} $}
\put(-74,65){\large $g_{21} $}
\put(-111,65){\large $g_{20} $}

\put(81,3){\large $g_{01} $}
\put(60,22){\large $g_{02} $}
\put(91,15){\large $g_{03} $}
\put(149,3){\large $g_{10} $}
\put(139,20){\large $g_{12} $}
\put(171,17){\large $g_{13} $}
\put(118,89){\large $g_{23} $}
\put(123,73){\large $g_{21} $}
\put(91,75){\large $g_{20} $}
\put(189,89){\large $g_{32} $}
\put(203,72){\large $g_{31} $}
\put(171,75){\large $g_{30} $}
\end{picture}
\end{center}
\vskip -1.4pc
\caption{A geometric description of generators of the truncated 2-, 3-simplexes with $G$-labels. 
\label{fig.colo3}}
\end{figure}

In addition, let us examine the case where $G_{\mathbb{C}}$ is $PSL_2(\mathbb{C})$ and every $K_{\mathbb{C},\ell}$ is conjugate to the unipotent subgroup
such that $K_{\mathbb{C},s} \cap K_{\mathbb{C},t } = \{1_{G_{\mathbb{C}}} \}$ for $s \neq t. $
Then, we have the quandle $X'$, from Example \ref{rei322}, as the union of the quandles $ \sqcup^{\#L }_{\ell = 1} K_{\mathbb{C},\ell} \backslash G_{\mathbb{C}}= \sqcup^{\#L }(\mathbb{C}^2 \setminus \{ 0,0\})/ \sim $.
We will describe a quasi-inverse $\beta$ mentioned in Proposition \ref{tukau}.
For this, consider the following subcomplex of $ C_{n}^{\Delta}(X_{\mathbb{C}};\Z) $:
$$ C_{n}^{h \neq}(X_{\mathbb{C}}) = \Z \langle \ [ (a_0 , b_0), \dots, (a_n,b_n) ]\in C_{n}^{\Delta}(X_{\mathbb{C}};\Z ) \ \bigr| \ a_i b_j \neq a_j b_i \mathrm{ \ for \ any \ } i, \ j\mathrm{ \ with \ } i \neq j \ \bigr\rangle. $$
Then, this complex is known to be an acyclic $\Z[G_{\mathbb{C}}]$-free complex.
Consider the correspondence
\[ g_{ij} :(X_{\mathbb{C}})^{n+1} \lra G; \ \ \ \bigl( (a_0 , b_0), (a_1, b_1),\dots, (a_n,b_n)\bigr) \longmapsto \left(
\begin{array}{cc}
a_i & b_i \\
a_j/ (a_i b_j - a_j b_i )& b_j/ (a_i b_j - a_j b_i)
\end{array}
\right). \]
This gives rise to a homomorphism
\[ \beta: C_{n}^{h \neq}(X_{\mathbb{C}}) \lra \overline{C}_n(G_{\mathbb{C}}, \mathcal{K}_{\mathbb{C}}).\]
Then, Zickert \cite[\S 3]{Zic} (see also \cite[Corollary 9.6]{NM}) showed that this $\beta$
is a chain map and a $\Z[G]$-homomorphism.
To summarize, this $\beta$ gives a quasi-inverse of the chain map $ \alpha : C_n^{\rm gr}(G_{\mathbb{C}},\mathcal{K} ) \ra C_{n}(X_{\mathbb{C}}) .$

We return to the discussion of a hyperbolic link $L$, and state Theorem \ref{class} below.
Fix a diagram $D$ of $L.$
Then, we have the holonomy representation $ \rho: \pi_L \ra PSL_2( \mathbb{C})$.
As is well-known, $\rho$ is injective. Thus, it is more reasonable to use matrices in $PSL_2( \mathbb{C})$,
than to use (Wirtinger) group presentations of $\pi_L$.
Here, we should mention the following lemma obtained from hyperbolicity.
\begin{lem}[{see \cite[\S 5]{Zic} or \cite[Lemma 7.2]{IK}}]\label{h3131i}
Let $ \sigma \in C_{3}^{\Delta }(X_{\mathbb{C}} ;\Z ) $ be a 3-cycle which represents the fundamental 3-class $\rho_*(E_L, \partial E_L)$ of a hyperbolic link.
Then, this 3-cycle $\sigma $ lies in the subcomplex $ C_{3}^{h \neq}(X_{\mathbb{C}} )$.
\end{lem}
In summary, we obtain the conclusion:
\begin{thm}\label{class}
Let $L$ be a hyperbolic link with the holonomy representation $ \rho: \pi_L \ra PSL_2( \mathbb{C})$.
Let $G $ be the image $\rho(\pi_L ) $, and $\mathcal{K} $ be the subgroups $ \rho (\partial \pi_L)$.
Choose a diagram $D$, and take the quandle 3-class $[\mathcal{S}_\rho]$; see \S \ref{s112}.
Then, the following 3-cycle represents the fundamental 3-class in $H_3( E_L , \partial E_L ;\Z) \cong H_3^{\rm gr}( \pi_L, \partial \pi_L;\Z) \cong \Z.$
$$ \mathrm{res}(\beta) \circ \varphi_3 ([\mathcal{S}_\rho])\in \overline{C}_3(G, \mathcal{K};\Z). $$
\end{thm}

\subsection{Example; the figure eight knot. }\label{238}

As the simplest case, we let $L$ be the figure eight knot as in Figure \ref{ftf}.
By the Wirtinger presentation, we have
$$\pi_1(S^3 \setminus L) \cong \langle g,h \ | \ h^{-1} gh = g^{-1} h^{-1}ghg^{-1} hg \rangle , $$
where $ g$ and $h$ are meridians derived from the arcs $\alpha_1$ and $\alpha_2$, respectively.
We denote the two classes in $Q_L$ of $g$ and $h \in \pi_L $ by $a$ and $b$, respectively.
Then, by definition, the fundamental 3-class $[\mathcal{S}_{\mathrm{id}_{Q_L}}] $
is given by
$$ - (b \lhd a,a, b )-(b \lhd a,b, a )+((b \lhd a) \lhd b,a, a\lhd b )+( b,b,b \lhd a ) \in C_3^Q( Q_L;\Z). $$
Notice that the first term is sent to zero by $\varphi_*$. Then, $\varphi_*[\mathcal{S}_{\mathrm{id}_{Q_L}}] $ is computed as
\vskip -2.1pc
{\normalsize
\begin{align*}
\lefteqn{} \\
& - ( b \lhd a,a, b ,p ) + ((b \lhd a) \lhd a ,a , b ,p ) + ((b \lhd a) \lhd b ,a \lhd b , b ,p ) - ((( b \lhd a) \lhd a)\lhd b ,a \lhd b , b ,p)\\
& - ( b \lhd a,b, a ,p ) + ((b \lhd a) \lhd b ,b ,a ,p) + ((b \lhd a) \lhd a ,b \lhd a, a ,p ) - ((( b \lhd a) \lhd b)\lhd a ,b \lhd a , a ,p ) \\
& +((b \lhd a) \lhd b,a, a\lhd b ,p ) -(((b \lhd a) \lhd b )\lhd a ,a, a\lhd b ,p) -(((b \lhd a) \lhd b) \lhd (a\lhd b ),a\lhd (a\lhd b ), a\lhd b ,p ) \\
& + ((((b \lhd a) \lhd b) \lhd a )\lhd (a\lhd b ),a\lhd (a\lhd b ), a\lhd b ,p).
\end{align*} }
As matters now stand, they seem complicated.

Thus, following Theorem \ref{class}, we consider the well-known holonomy representation $ \rho :\pi_L \ra PSL_2(\Z[\omega])$,
where $\omega$ is $ (-1 + \sqrt{-3})/2$.
This is represented by the $X_{\mathbb{C}}$-coloring $\mathcal{C}$ in Figure \ref{ftf}. 
Accordingly, if we replace $a$ by $(1,0)$ and $b$ by $(0, \omega)$, the 3-cycle $\varphi_*[\mathcal{S}_{\mathrm{id}_{Q_L}}] $ above is reduced to \vskip -2.1pc
{\normalsize
\begin{align*}
\lefteqn{} \\
& - ( (-\omega ,\omega ),(1,0) , (0,\omega),p \bigr) +\bigl((-2\omega, \omega),(1,0) , (0,\omega) ,p\bigr) + \bigl((-\omega ,\omega+1 ),(1,\omega-1 ) , (0,\omega) ,p\bigr)\\
& - \bigl(( -2 \omega, \omega+2 ) ,(1,\omega -1 ), (0,\omega) ,p \bigr)\\
& - ( (-\omega ,\omega ), (0,\omega),(1,0) ,p\bigr)+ \bigl( (-\omega ,\omega +1 ) , (0,\omega) , (1,0) ,p \bigr) + \bigl( (-2 \omega ,\omega ) ,(-\omega, \omega) , (1,0) ,p ) \\
& - \bigl((-2\omega-1, \omega +1 ) ,(-\omega, \omega) ,(1,0) ,p\bigr) \\
& +\bigl((-\omega, 1+\omega) ,(0. \omega ), (1, \omega-1) ,p\bigr) - \bigl((-\omega, 2+\omega) ,(0. \omega ), (1, \omega-1) ,p\bigr) \\
& -\bigl( (-2\omega-1, 4),(-\omega, 1+\omega) , (1, \omega-1) ,p \bigr) + \bigl( (-2\omega-2, 6 -\omega ),(-\omega, 1+\omega) , (1, \omega-1) ,p \bigr).
\end{align*}}
Hence, for example, if $p=(0,1) $ and we apply the composite $\beta $ to this cycle,
we can describe explicitly the fundamental 3-class by Theorem \ref{class}. However,
the description forms long; we omit writing it.

\vskip -0.185pc
\begin{figure}[htpb]
\begin{center}
\begin{picture}(10,60)
\put(-128,42){\LARGE $D $}
\put(-125,13){\large $\alpha_2 $}
\put(-61,51){\large $\alpha_1 $}
\put(-61,13){\large $\alpha_3 $}
\put(-44,-5){\large $\alpha_4 $}

\put(24,46){\large $\mathcal{C}(\alpha_1):= (1,0) $}
\put(24,29){\large $\mathcal{C}(\alpha_2):= (0,\omega ) $}
\put(24,12){\large $\mathcal{C}(\alpha_3):= (1 , \omega-1)$}
\put(24,-5){\large $\mathcal{C}(\alpha_4):= (\omega,-\omega )$}
\put(-116,22){\pc{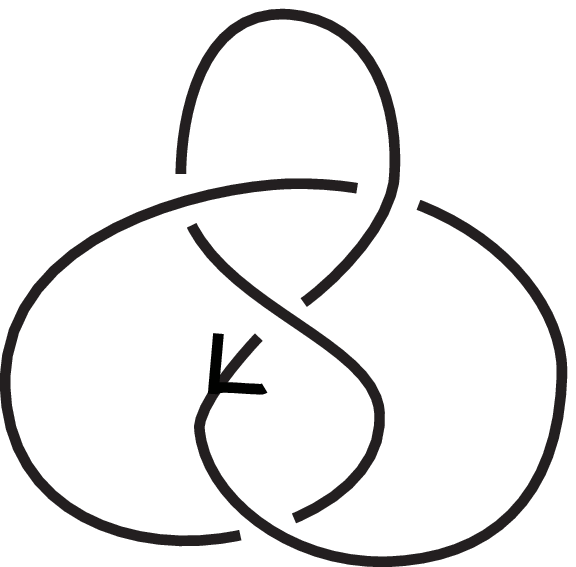}{0.43104}}

\end{picture}
\end{center}\caption{\label{ftf}The holonomy representation of $4_1$ as an $X$-coloring. Here $\omega =(-1 + \sqrt{-3})/2$. }
\end{figure}

\section{Proofs of the main theorems.}\label{s33}
We will prove the theorems in \S 2.
If $L$ is the trivial knot, the theorems are obvious. Thus, we may assume that $L$ is non-trivial in what follows.
We begin by proving Theorem \ref{mainthm1}.
Since the discussion is not functorial,
the proof may seem intricate.

\subsection{Proof of Theorem \ref{mainthm1}.}\label{2363344}

\begin{proof}Recall the complexes in Sections \ref{s11}--\ref{s112}.
Since they are 
functorial by construction,
we obtain a commutative diagram:
\begin{equation}\label{aa}{\normalsize
\xymatrix{
C_*^R (Q_L ; \Z ) \ar[rr]^{\varphi_* } \ar[d]_{ f_* }& &
C_*^{\Delta } (Q_L ; \Z ) \ar[d]_{ f_* } \ar@/_1.3pc/[rr]_{\beta} & & C_*^{\rm gr } (\pi_L ,\partial \pi_L ; \Z ) \ar[d]_{ f_* } \ar[ll]_{\alpha }
\\
C_*^R (X; \Z ) \ar[rr]^{\varphi_* } & &
C_*^{\Delta } (X ; \Z ) & & C_*^{\rm gr } (G ,K ; \Z ) \ar[ll]^{\alpha }.
}}\end{equation}
Since $ L$ is either a prime non-cable knot or a hyperbolic link by assumption,
the right $\beta $ comes from the quasi-isomorphism $\alpha$ in Examples \ref{tuu} and \ref{knotmal}.



We will explain \eqref{kaa2} below.
Recall the quandle 3-class $[\sh_{\mathrm{id}_{Q_L}}]$ from Example \ref{r11}.
Then, denoting by $[\pi_L, \partial \pi_L ]$
a generator of $H_3^{\rm gr} (\pi_L, \partial \pi_L ) \cong \Z $, the diagram \eqref{aa} admits $N_L \in \mathbb{Z} $ such that
$$ \beta \circ \varphi_3 ([\sh_{\mathrm{id}_{Q_L}} ]) = N_L [\pi_L, \partial \pi_L ] \in H_3^{\Delta } (Q_L ;\Z ) \cong \Z. $$
Then, for every group relative 3-cocycle $\theta$, setting $ \phi_{ \theta}= (\beta \circ \varphi_3)^* \circ f^* ( \theta) $ yields the equalities
\begin{equation}\label{kaa2} \langle \phi_{ \theta} , \ [\sh_{\mathrm{id}_{Q_L}} ] \rangle=
\langle f^* (\theta) , \ \beta \circ \varphi_3 ([\sh_{\mathrm{id}_{Q_L}} ] )\rangle= N_L \langle f^* (\theta) , \ [\pi_L, \partial \pi_L] \ \rangle \in A_G.
\end{equation}
Hence, it is enough to show $N_L= \pm1$. Actually, if $N_L= \pm1$,
we canonically obtain a map $ \phi_{ \theta}: \mathrm{Reg}_D \times (\mathrm{Arc}_D)^2 \ra A$ from the definition of $[\sh_{\mathrm{id}_{Q_L}} ]$,
which justifies the desired equality \eqref{11}.

If $L $ is a hyperbolic link, and $f$ is the associated holonomy $\pi_L \ra P SL_2(\C)$ with $\theta =\mathrm{CS} $, we immediately have $N_L= \pm1$ by \eqref{kaa}.

It remains to work in the case where $L$ is a knot which is neither hyperbolic nor cable.
Then, thanks to the JSJ-decomposition \eqref{13882} above,
there is a solid torus $V_1 \subset S^3$, such that $V_1 $ contains the link $L$ and $V_1 \setminus L$ is either hyperbolic or reducible.

In the former case, when $V_1 \setminus L$ is hyperbolic, we now give a diagram (\ref{aa22}) below.
Regarding $V_1 \setminus L$ as a hyperbolic link, $L'$, in the 
3-sphere, we denote by $K_2$ another subgroup of $\pi_1 ( S^3 \setminus L' )$ arising from $\partial ( V_1) $.
Then, the link quandle $Q_{L'} $ is bijective to $ K_1 \backslash \pi_{L'}\sqcup K_2 \backslash \pi_{L'}$ by definition.
We consider the homogenous quandle of the form $ K_1 \backslash \pi_{L}\sqcup K_2 \backslash \pi_{L} $, and denote it by $Q_W$.
Then, the inclusion $j: S^3 \setminus L' \hookrightarrow S^3 \setminus L $ defines a quandle map $j_*: Q_{L'} \ra Q_W$, and we take a canonical injection $Q_L \hookrightarrow Q_W$.
In summary, we have the commutative diagram on the third homology groups:
\begin{equation}\label{aa22}{\normalsize
\xymatrix{
H_3^Q (Q_L ; \Z ) \ar[rr]^{i_*^Q } \ar[d]_{ \varphi_* } & &
H_3^Q (Q_W ; \Z ) \ar[d]_{ \varphi_* } & & H_3^Q (Q_{L' } ; \Z ) \ar[ll]_{j_*^Q } \ar[d]_{ \varphi_* }
\\
H_3^{\Delta } (Q_{L} ; \Z ) \ar[rr]^{i_* } & &
H_3^{\Delta } (Q_{W} ; \Z )& & H_3^{\Delta } (Q_{L'} ; \Z ) \ar[ll]_{j_* } .
}}\end{equation}
In the appendix (Lemma \ref{4321}), we later show the isomorphisms $ H_3^Q (Q_W ; \Z ) \cong H_3^Q (Q_{L' } ; \Z ) \cong \Z^2 $,
such that the matrixes of $i_*^Q$ and $j_*^Q$ are given by $(1,0)$ and $ {\small \left(
\begin{array}{cc}
1 & 1 \\
1 & 1 \end{array}
\right) }$, respectively. Hence, since the right $ \varphi_*$ is surjective by the former discussion,
the commutativity with \eqref{13882} implies $N_L=\pm 1$ as desired.

Finally, we discuss the case where $V_1 \setminus L$ is reducible.
For this, put a hyperbolic link $L_0 $ in $D^2 \times S^1$, and attach it to $V_1 \setminus L$.
Then, the situation is reduced to the above case. Hence, the proof completes.
\end{proof}
\begin{proof}[Proof of Theorem \ref{clA1ee22}] Since $L$ is neither composite nor cable,
the above discussion readily implies the bijectivity of $\varphi_3 $.
\end{proof}

\begin{rem}\label{ssss}
We mention the assumption of hyperbolicity.
As a counter example, consider the Hopf link $L$.
Since $\pi_1(S^3 \setminus L) \cong \Z^2$ and the boundary inclusions induce isomorphisms on $\pi_1$,
the link quandle $Q_L$ consists of two points. Hence, the chain map $\varphi_*$ is zero by definition.
However, the pairing $ \langle \theta , \ f_* [E_L, \ \partial E_L] \rangle $ is not always trivial. In summary, it is seemingly hard to generalize the theorem \ref{mainthm1} in every link case.


\end{rem}
\subsection{Proof of Theorem \ref{mainthm133}; Malnormality and Transfer.}\label{2364334}

Next, turning to malnormality and transfer, we will complete the proof of Theorem \ref{mainthm133}.
For this, we shall mention
a key proposition obtained from transfer.
\begin{prop}[{cf. Transfer; see \cite[\S III.10]{Bro}}]\label{m44}
Let $K_1, \dots, K_{\# L}$ be finite subgroups of $G$, and let $Y$ be $ \sqcup_i (K_{i } \backslash G)$ with action.
Assume that 
all the order $|K_i|$ is invertible in the coefficient group $A$.
Then, the chain map $$\alpha :( \mathcal{P}_* \otimes_{\Z} A , \partial_*) \lra ( C_*^{\Delta}(Y ) \otimes_{\Z} A , \partial_*) $$ with coefficients $A$ is
a quasi-isomorphism.
\end{prop}
\begin{proof} According to the same discussion on the transfer; see \cite[\S III.9--10]{Bro}.
\end{proof}

\begin{proof}[Proof of Theorem \ref{mainthm133}]
First, we assume the assumption (i), that is, $ \mathcal{K} \subset G$ is malnormal.
Then, Proposition \ref{tukau} again ensures a quasi-inverse $\beta': C_*^{\Delta } (X )_G \ra C_*^{\rm gr } (G , \mathcal{K} )$, where $X= \sqcup_{i \in I}K_i \backslash G $ as in Example \ref{exaqSGH}.
Then, for any group 3-cocycle $\theta $, we set $ \phi_{ \theta}= (\beta' \circ f_* \circ \varphi_3)^* ( \theta) \in A $ as a quandle 3-cocycle,
and the group 3-cocycle $\theta$ is represented by a map $X^4 \ra A$.

We will show the equality below. Let $L$ be a hyperbolic link or a prime non-cable knot,
which ensures the 3-class $[\pi_L, \partial \pi_L ] $ by the malnormal property of $L$.
Compute the pairing $ \langle \phi_{ \theta} , [\mathcal{S}_f] \rangle$ as 
\begin{equation} \langle \phi_{ \theta} , [\mathcal{S}_f] \rangle = \langle
f^* \circ (\beta' )^* ( \theta) , \ (\varphi_3)_* [\mathcal{S}_f] \rangle =
\langle f^* \circ (\beta' )^* ( \theta) , \ (\alpha)_* [\pi_L, \partial \pi_L ] \rangle \notag \end{equation}
\begin{equation}\label{bb111}
\ \ \ \ \ \ \ \ \ \ \ \ \ \! = \langle \alpha^* \circ f^* \circ (\beta' )^* ( \theta) , \ [\pi_L, \partial \pi_L ] \rangle= \langle f^* ( \theta) , \ [\pi_L, \partial \pi_L ] \rangle .
\end{equation}
Here, the second equality is obtained by \eqref{kaa2} and Proposition \ref{tukau}, and the others are done by functoriality.
From the definition of $ [\mathcal{S}_f]$ and $ \phi_{ \theta} $, this equality \eqref{bb111} is equivalent to the desired statement.

Finally, we change to the assumption (ii). By Proposition \ref{m44}, we can similarly get a quasi-inverse $\beta$ of $\alpha $ in the coefficient group $A$.
Thus, the required equality is done by the same discussion as \eqref{bb111}; hence, the desired statement also holds in (ii).
\end{proof}

\subsection{Examples of 3-cocycles. }\label{exex}
Finally, 
we will point out that, in some cases, such group 3-cocycles $\theta $ have much simpler expressions in terms of quandle cocycles.
We end this section by giving two examples.
\begin{exa}[cf. \cite{Nos3}]\label{ma33ee1}
First, we will observe some triple Massey products.
This example is essentially due to \cite[\S 4.2]{Nos3}.
Regard the finite field $\F_q$ as the abelian group $ (\Z/p) ^m$, where $q=p^m$ and $p\neq 2$.
Consider the (nilpotent) group on the set
\[ G:= \Z/2 \times \F_q \times ( \F_q \wedge_{\Z} \F_q ) , \]
with operation
\begin{equation}\label{clauwens} (n, a , \kappa) \cdot (m, b, \nu )= (n+m, \ (-1)^m a +b, \ \kappa + \nu + [(-1)^m a \otimes b]). \ \end{equation}
Letting the subgroup $ K$ be $\Z/2 \times \{ 0 \} \times \{ 0 \} \subset G$ and $x_0 \in K$ be $(1,0,0)$, we have the quandle of the form $X= \F_q \times ( \F_q \wedge_{\Z} \F_q ) $.

The cohomology $H^3(G;\F_q) $ is complicated. 
In fact, the cohomology has 3-cocycles $ \theta_{\Gamma}$, which are derived from triple Massey product (see \cite[Proposition 4.8]{Nos3}).
However, the author \cite[Lemma 4.7]{Nos3} showed that
the pullback $\varphi_3^* \theta_{\Gamma}:X^3 \ra \F_q$ is formulated as
$$ (\varphi_3^* \theta_{\Gamma}) \bigl((x, \alpha), \ (y, \beta),\ (z, \gamma) \bigr) = (x -y)^{q_1} (y-z)^{q_2+q_3 } z^{q_4}. $$
with some prime powers $q_1 , q_2,q_3,q_4 \in \Z$.
Since this formula is relatively simple, we can compute the relative fundamental 3-class in an easier way than the grouptheoretic method; 
see \cite[\S 5]{Nos3} for a computation.
\end{exa}

\begin{exa}\label{3333ee1}
Next, we focus on the quandles dealt in the paper \cite{Nos3}, and show Proposition \ref{oo24} below. The result will be useful to show theorems in \cite{Nos9}.

Let $L$ be a prime knot or a hyperbolic link, and $\mathcal{O}_{\ell}$ be the set $\{ g^{-1} \mathfrak{m}_{\ell} g\}_{g \in \pi_L}$.
Fix, the inclusion from the link quandle $Q_L$ into $ s\sqcup^{\# L}_{\ell =1} \mathcal{O}_{\ell} $ which sends $K_{\ell} g $ to $ g^{-1} \mathfrak{m}_{\ell} g. $
Given a right $\Z[\pi_L ]$-module $M $,
we have the semi-direct product $G:= M \ltimes \pi_L$.
Fix $b_1, b_2, \dots, b_{\# L} \in M$.
Then, for $\ell \leq \# L$, we fix $b_\ell \in M$,
and consider the subgroup 
$$K_{\ell } := \bigl\{ \bigl( b_\ell (1-\mathfrak{m}_\ell ^s \mathfrak{l}_\ell ^t ) , \ \mathfrak{m}_\ell ^s \mathfrak{l}_\ell ^t \bigr) \in M \ltimes \pi_L \bigl| \ s,t \in \Z^2 \ \bigr\} .$$
Then, we have a quandle $ X= \sqcup_{i \leq \#L } K_{\ell} \backslash (M \ltimes \pi_L )$.
Then, according to the above inclusion, $X$ is bijective to $ M \times \sqcup^{\# L}_{\ell =1} \mathcal{O}_{\ell}$,
and the quandle structure is equivalent to
$$ (a ,g )\lhd (b, h )= ((a-b) h +b, h^{-1}gh )$$
for $a,b \in M$ and $g, h \in \sqcup^{\# L}_{\ell =1} \mathcal{O}_{\ell}$.
We notice two lemmas:

\begin{lem}\label{oo2388}
If the pair $(\pi_L , \partial \pi_L)$ is malnormal,
so is the pair $ K_{\ell } \subset M\rtimes G$.
\end{lem}

\begin{lem}\label{oo23}
Take a $G$-invariant multi-linear map $ \psi: M^n \ra A$, where $G$ trivially acts on $A$.
Let us identify the quandle on $X$ with that on $ $.
Then, the following map from the normalized complex (see Example \ref{tukau5}) is an $n$-cocycle. 
$$\Psi: C_n^{\rm Nor}(X)_{\As(X)} \lra A ; \ \ \ ((a_0,g_0) \dots, ( a_n,g_n)) \longmapsto
\psi (a_0- a_1, a_1 -a_2,\dots, a_{n-1} -a_n). $$
\end{lem}
Then, it is sensible to consider the pullbacks of $\varphi$ and $\alpha$ in the diagram \eqref{aa}. Here we focus on the third case $n=3$.
Then, as in \cite[Theorem 5.2]{Nos2}, we can easily see that 
the pullback of the IK map $\varphi$ is formulated by
$$ \varphi^*_3 (\Psi) ( (a_1, z_1 ), (a_2, z_2 ), (a_3, z_3 )) = \psi \bigl( (a_1 -a_2 )(1- z_2) , a_2 -a_3, \ a_3 -a_3 z_3^{-1}) . $$
Moreover, consider another pullback by $\alpha $,
where we use the expression of $\alpha $ in Example \ref{tukau5}.
Notice that the projection $M \rtimes G \ra M \times \mathcal{O}_{\ell}$,
is equal to the map which takes $(m,g)$ to $ (y_{\ell} g + m, g^{-1}\mathfrak{m}_{\ell}g )$.
Then, in terms of the non-homogenous complex in Example \ref{rei122}, the pullback, $\alpha^* (\Psi)$, is represented by
\begin{equation}\label{ff4} \theta_{\ell} : \bigl( (a_1, g_1 ), (a_2, g_2 ), (a_3, g_3 )\bigr) \longmapsto
\psi \bigl( (a_1 +y_\ell -y_\ell \cdot g_1) g_2 g_3 , (a_2+y_\ell -y_\ell \cdot g_2) \cdot g_3 , \ a_3+b_\ell -b_\ell \cdot g_3 ) ,\end{equation}
and $\eta_{\ell} =0.$
Combing with Theorems \ref{mainthm1} and \ref{mainthm5}, we readily have
\begin{prop}\label{oo24}
As in Theorem \ref{mainthm1}, let $L$ be either a hyperbolic link or a prime knot which is not a cable knot or a torus knot. 
Then, the pairing of the 3-cocycle $ \alpha^* (\Psi) $ in \eqref{ff4} and $[Y_L, \partial Y_L]$ is equal to the quandle cocycle invariant $ \langle \varphi^*_3 (\Psi) , [\mathcal{S}]\rangle$.

Furthermore, if $K$ is the $(n,m)$-torus knot, the same equality holds modulo $nm$.
\end{prop}
\end{exa}

\section{Theorem for cable knots. }\label{2734}
While Theorems \ref{mainthm1} and \ref{mainthm133} assumed non-cabling knots,
this section focuses on cable knots.
To simplify the study, consider a solid torus $V \subset S^3$ such that
$ V \setminus L$ is the $(m,n)$-torus knot.
By the formula \eqref{13882} from the JSJ-decomposition,
it is sensible to consider either the torus knot $T_{m,n}$ in $S^3$ or the solid one in $V$.
We denote the latter by $S_{m,n}$;
Regarding $ S^3 \setminus ( L \sqcup (S^3 \setminus V)) $ as a link complement,
the knot $S_{m,n}$ admits a link-diagram in $\R^2$; see Figure \ref{fig3lo3}.

\begin{figure}[htpb]
\begin{center}
\begin{picture}(100,110)
\put(117,48){\pc{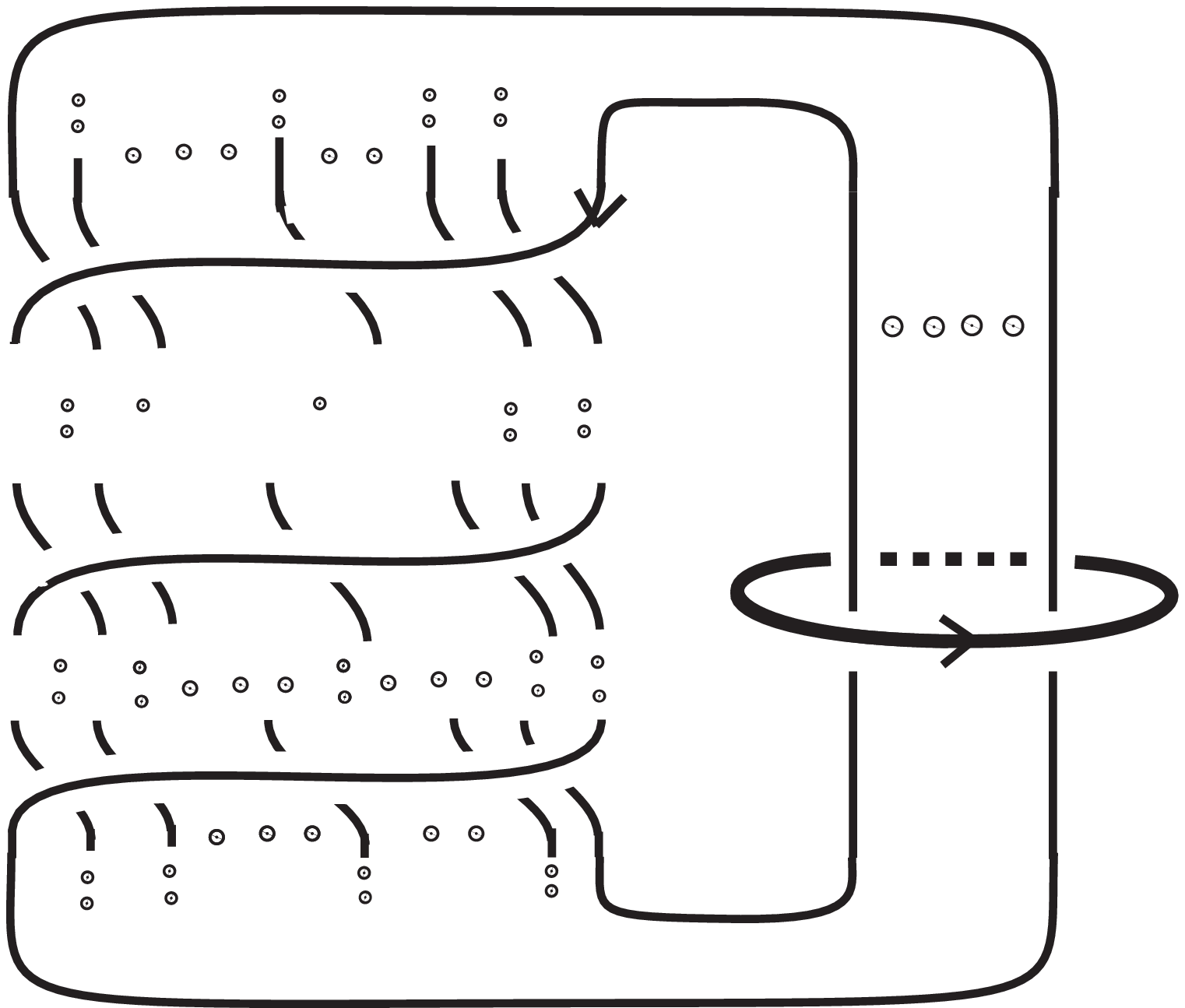}{0.25}}
\put(-150,49){\pc{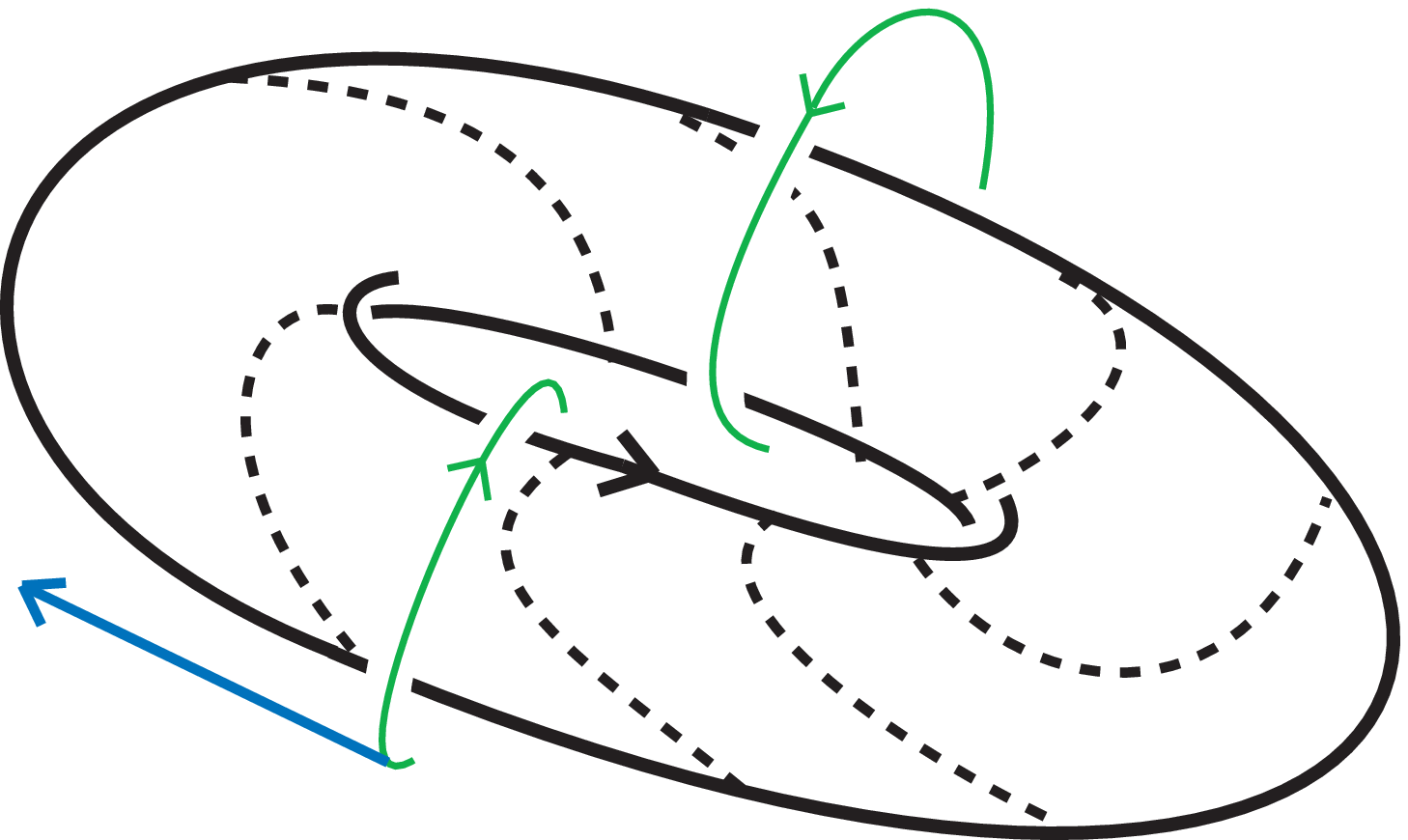}{0.31}}
\put(237,71){\Large $\subset \R^2 $}
\put(-21,71){\Large $\subset S^3 $}
\end{picture}
\end{center}
\vskip -1.737pc
\caption{The torus knot $S_{m,n }$ in the solid torus, and the link diagram on $\R^2$. 
\label{fig3lo3}}
\end{figure}


As in the theorems \ref{mainthm1} and \ref{mainthm133}, 
we get similar results modulo some integer. Precisely, 
\begin{thm}[cf. Theorems \ref{mainthm1} and \ref{mainthm133}. See \S \ref{229123} for the proof.]\label{mainthm5}
Assume that $L$ is either the torus knot $T_{m,n}\subset S^3$ or that $S_{m,n} \subset V$.
Let $N =n $ if $L$ is $S_{m,n} $, and let $N=mn$ if $L$ is $T_{m,n}. $

Then, for any relative group 3-cocycle $ \theta \in H^3(G,\mathcal{K};A)$,
there is a map $ \phi_{\theta }: \mathrm{Reg}_D \times \mathrm{Arc}_D \times \mathrm{Arc}_D \ra A $
for which the following equality holds in the coinvariant $A_G$:
\begin{equation}
\label{1122} \langle \ \theta , \ f_*[E_L, \ \partial E_L ] \rangle = \langle \phi_{\theta}, [\mathcal{S}_f] \rangle \in A_G, \ \ \ \ \mathrm{modulo\ the \ integer \ } N. \end{equation}

Furthermore, if the pair $(G,\mathcal{K} )$ is malnormal,
then there are a quandle 3-cocycle $ \phi_{\theta }$ and an $X$-coloring $\mathcal{S}_f $
such that $\langle \ \theta , \ f_*[E_L, \ \partial E_L ] \rangle = \langle \phi_{\theta}, [\mathcal{S}_f] \rangle $ modulo the integer $N $.
\end{thm}
In conclusion, we have a diagrammatic computation for cable knots, 
although the statement is considered modulo some integer.
In addition, we later see that the discussion to show \eqref{1122} is reduced to the homology of cyclic groups;
hence, the pairing modulo $N$ does not have more information than cyclic groups.


\subsection{Observations on the torus knots. }\label{23dd11123}

Before going to the proof, we show two lemmas, and observe an essential reason why the statement is considered modulo $N$.

Consider the $(n,m)$-torus knot $L$ in the 3-sphere $S^3$. Let $G$ be $\pi_1(S^3 \setminus L)$, and let $K$ be the pherihedral subgroup.
Fix $(a,b,n,m)\in \Z^4$ with $an +bm=1.$
According to \cite[\S 2]{HWO}, we mention these group presentations
$$ \pi_1(S^3 \setminus T_{n,m}) \cong \langle \ x,y, \ | \ x^n =y^m \ \rangle \ \supset \ \langle x^a y^b , \ (x^a y^b )^{-nm } x^n \rangle=K . $$
Then, Theorem \ref{tuu333} says that the pair $ (G,K) $ is not malnormal.
However, regarding the center $Z = \langle x^n \rangle \subset G$,
the quotients $ G/Z$ and $K/Z$ are isomorphic to the free product $\Z/n * \Z/m$
and $\Z$, respectively.
Denote the quotients $ G/Z$ and $K/Z$ by $\mathcal{G}$ and $\mathcal{K}$, respectively.
Then, it is known \cite[\S 2]{HWO} that the pair $ (\mathcal{G},\mathcal{K})$ is malnormal.
We will show the following:
\begin{lem}\label{h3ai} For $* \geq 2$, there are isomorphisms
$$H_*^{\Delta}(K \backslash G;\Z ) \cong H_*^{\Delta}(\mathcal{K} \backslash \mathcal{G};\Z ) \cong H_*^{\rm gr}( \mathcal{G}, \mathcal{K};\Z) \cong \left\{\begin{array}{ll}
\Z , &\quad \mathrm{if \ \ } *=2, \\
\Z/nm , &\quad \mathrm{if \ \ } * {\rm \ is \ odd, \ and \ }*\geq 3, \\
0, & \mathrm{otherwise}.
\end{array}
\right. $$
\end{lem}
\begin{proof} The first one is obtained by the settheoretic equality $ K \backslash G=\mathcal{K} \backslash \mathcal{G}$. The second is done by the malnormality.
We explain the last one: by a Mayer-Vietoris argument, the homology of $\Z/n * \Z/m$ is
that of the pointed sum $ L^{\infty}_n \vee L^{\infty}_m $, where $L^{\infty}_m $ is the infinite dimensional lens space with fundamental group $\Z/m$.
Hence, the long exact sequence \eqref{g666} readily leads to the conclusion.
\end{proof}

In summary, since the proofs in this paper often employ the Hochschild homology $H_*^{\Delta}(K \backslash G )$,
it is reasonable to consider the pairing modulo $nm$.

In addition, we similarly observe another knot $S_{n,m}$ in the solid torus, in details.
Fix four integers $(n,m,a,b)\in \Z^4$ with $an+bm =1$.
Consider the following subspace
$$ \bigl\{ \ \ (z, w) \in \mathbb{C}^2\ \ \bigl|\ \ |z|^2 +|w|^2=1, \ \ \ \ |z^{n} +w^m| < \frac{1}{n^m m^n}, \ \ \ |z|<\frac{1}{3} \ \ \bigr\}. $$
We can easily check that
the space is homeomorphic to the $(n,m)$-torus knot $V\setminus T_{n,m}$.
Since the space is regarded as a restriction of a Milnor fibration over $S^1 $, it is an Eilenberg-MacLane space.
Furthermore,
as in the usual computation of $\pi_1 (S^3 \setminus T_{n,m}) $,
set up the two subsets
$$ U_1:= \bigl\{ \ (z,w) \in V\setminus T_{n,m} \ \bigl| \ \ |z|^2 \leq 1/2 \ \bigr\}, \ \ \ \ \ U_2:= \bigl\{ \ (z,w) \in V\setminus T_{n,m} \ \bigl| \ \ |z|^2 \geq 1/2 \ \bigr\}. $$
Since $U_1 \simeq S^1$ and $U_2 \simeq S^1 \times S^1$, a van-Kampen argument (see \cite[\S 15]{BZH}) can conclude
$$ \pi_1 (V\setminus T_{n,m}) \cong \langle \ x,z' \ | \ x^m z' =z'x^m \ \rangle .$$
Here, the longitude $ \mathfrak{l}$ is represented by $ \mathfrak{m}^{-nm } x^n $, as before. 
In summary, we can easily obtain the following:
\begin{lem}\label{h3ai}
\begin{enumerate}[(i)]
\item
The center of $ \pi_1 (V\setminus T_{n,m})$ is generated by $x^m$ and is isomorphic to $\Z$.
\item The quotient group of $ \pi_1 (V\setminus T_{n,m})$ subject to the center is isomorphic to
$ \Z/m * \Z$.
\item The subgroup generated by the meridian $ [\mathfrak{m}]$ is malnormal, and is isomorphic to
$\Z$.
\end{enumerate}
\end{lem}
Since $S_{n,m}$ is embedded in $S^3$, we have the link quandle $Q_V$.
As a result, 
we similarly have
$$H_k^{\Delta}(Q_V;\Z ) \cong H_k^{\Delta}(K \backslash G;\Z ) \cong H_k^{\rm gr}( \Z/m * \Z , \Z,\Z ;\Z) \cong \left\{\begin{array}{ll}
\Z , &\quad \mathrm{if \ \ } k=2, \\
\Z/m , &\quad \mathrm{if \ \ } * {\rm \ is \ odd, \ and \ }*\geq 3.\\
0 , &\quad \mathrm{ \ \ } {\rm \ \ otherwise}.
\end{array}
\right. $$

Finally, we mention that
the inclusion $ j:V \setminus T_{n,m} \hookrightarrow S^3 \setminus T_{n,m} $ induces
$H_3^{\Delta} ( Q_V) \ra H_3^{\Delta} (\mathcal{K} \backslash \mathcal{G}) $ as an injection $\Z/ m \ra \Z/mn$.

\subsection{Proof of Theorem \ref{mainthm5}.}\label{229123}
We will give the proof of Theorem \ref{mainthm5}, based on the discussion in \S \ref{23dd11123}.

\begin{proof}
In this proof, we always deal with homology in torsion coefficients $\Z/N.$
Since the latter part from malnormality can be proven in a similar way to Theorem \ref{mainthm1},
we will only show the former statement.

By the above computations of homologies, the chain map $\alpha$ modulo $N$ yields
an isomorphism on $H_*(\bullet; \Z/N )$, which gives a quasi-inverse $\beta: C_*^{\Delta } (Q_V ) \ra C_*^{\rm gr } (\pi_V ,\partial \pi_V) $.
Furthermore, we suppose a quandle homomorphism $ f : Q_{T_{m,n}} \ra X$ with $X= H \backslash G $.

Similarly to \eqref{aa}, we have the following commutative diagram by functoriality.
\begin{equation}\label{aa2233}{\normalsize
\xymatrix{
H_3^R (Q_V; \Z/N) \ar[rr]^{\varphi_* } \ar[d]^{ j_* }& &
H_3^{\Delta } (Q_V ; \Z/N )_{ \pi_V } \ar[d]^{ j_* } \ar@/_1.3pc/[rr]_{\beta} & & H_3^{\rm gr } (\pi_V ,\partial \pi_V ; \Z/N) \ar[d]^{ j_* }\ar[ll]_{\alpha }
\\
H_3^R (Q_{T_{m,n}} ; \Z/N) \ar[rr]^{\varphi_* } \ar[d]^{ f_* } & &
H_3^{\Delta } (Q_{T_{m,n}}; \Z/N )_{ \pi_L } \ar[d]^{ f_* } \ar@/_1.3pc/[rr]_{\beta} & & H_3^{\rm gr } (\pi_{T_{m,n}} ,\partial \pi_{T_{m,n}}; \Z/N ) \ar[d]^{ f_* } \ar[ll]_{\alpha }\\
H_3^R (X; \Z/N ) \ar[rr]^{\varphi_* } & &
H_3^{\Delta } (X ; \Z/N )_G & & H_*^{\rm gr } (G ,K ; \Z/N) \ar[ll]^{\alpha }.
}}\end{equation}
Here, by the above discussion, the middle $j_*$ is reduced to the injection $\Z/ m \ra \Z/ n m $.
Hence, if we show that all the left $ \varphi_*$'s are surjective, then the rest of the proof runs as in the proof of Theorems \ref{mainthm1} and \ref{mainthm133}.

To show the surjectivity, we now set up appropriate $X$ and $f$.
Although there are many choices of such $X$ and $f$ for the proof, this paper relies on some results in \cite{Nos3} as follows:
Take arbitrary prime $p$ which divides $m$.
Choose the minimal $k \in \N$ such that $n$ is not relatively prime to $1+p^k$.
We set up the semi-direct product $G= (\Z/p)^k \rtimes \Z/ (1+p^k)$ and the subgroup $K = \Z/ (1+p^k)$.
Then, we can easily construct a group homomorphism $f: \Z/m *\Z/n\ra G$ that sends the subgroup $\Z$ to $ K $,
and induces the injection $f_* : H_3( \Z/m *\Z/n;\Z/p)\ra H_3(G;\Z/p) $ on homology.
Thus, it is enough for the surjectivity of the left $ \varphi_*$'s to show that the left bottom map $\varphi_* : H_3^R (X; \Z/N ) \ra
H_3^{\Delta } (X ; \Z/N )_G$ is injective. 
However, by noticing that the binary operation on $X= (\Z/p)^k =\mathbb{F}_{p^k}$ is
$x\lhd y =\omega (x-y)+y$ for some $\omega \in \mathbb{F}_{p^k} \setminus \{ 0,1 \}$ 
the injectivity of $X$ is already shown in the previous paper \cite[Lemmas 4.5--4.6]{Nos3}, which studies the chain map $\varphi_*$ for the quandle operation.

As a parallel discussion, 
when we choose any prime $p$ which divides $n$, the same injectivity can be shown. 
To summarize, since such a $p$ is arbitrary, we have shown the surjectivity of the left $ \varphi_*$'s. Hence, we complete the proof.
\end{proof}
Finally, we give a corollary:
\begin{cor}\label{3gi}
Let $L$ 
be the $(m,n)$-torus knot, and let $E_L$ be the complement space.
Let $K $ be a malnormal subgroup of $G$.
Then, for any relative 3-cocycle $\theta$, the $\ell$-torsion part of the pairing $\langle \ \theta , \ f_*[E_L, \ \partial E_L ] \rangle$ is zero.
Here $\ell $ is either the prime number coprime to $nm$ or $\ell=0.$
\end{cor}
In fact, as seen in the above proof, the fundamental 3-class of the torus knot must factor through $ H_3^{\Delta } (Q_V ; \Z ) \cong \Z/mn$.
As a result, for example, if $\theta$ is the Chern-Simon 3-class as in \eqref{ki2n}, the free (volume) part of the pairing turns out to be zero.

\subsection*{Acknowledgments}
The author is greatly indebted to
Tetsuya Ito and Yuichi Kabaya for
useful discussions on quandle, malnormality, and hyperbolicity.
This work is partially supported by JSPS KAKENHI Grant Number 00646903.

\appendix
\section{Appendix; the third quandle homology of some link quandles. }\label{2343136}

The purpose is to determine the third quandle homology of some link quandles (Theorems \ref{3331} and \ref{333221}).
In what follows, we assume the terminology in \S \ref{s112}, and
we deal with only integral homology (so we often omit writing $\Z$)

{ \large \baselineskip=17pt
For the purpose, we adopt an approach on the basis of \cite[\S 8]{Nos7}.
Thus, we shall review rack spaces from a quandle $X$.
Consider the orbit decomposition $X= \sqcup_{i \in O(X)} X_i$ from the action of $\As(X)$ on $X$. For each orbit $i \in O_X$, we fix $x_i \in X_i$.
Let $Y$ be either $X_i$ or the single point with their discrete topology.
Then, let us consider a disjoint union $ \bigsqcup_{n \geq 0} \bigl( Y \times ([0,1]\times X)^n \bigr), $
with the following two relations:}
\normalsize
\[(y, t_1,x_1, \dots, x_{j-1},0,x_j,t_{j+1},\dots, t_n,x_n) \sim (y, t_1,x_1, \dots t_{j-1},x_{j-1},t_{j+1},x_{j+1},\dots, t_n, x_{n} ), \]
\[ \hspace*{0pc} (y,t_1,x_1, \dots, t_{j-1}, x_{j-1},1,x_j, t_{j+1},x_{j+1} , \dots, t_n,x_n) \sim 
(y \lhd x_{j}, t_1,x_1\tri x_{j}, \dots, t_{j-1}, x_{j-1}\tri x_{j},t_{j+1},x_{j+1} ,\dots, t_n,x_n). \]
\large
\baselineskip=16pt
\noindent
Then, the {\it rack space} $B(X,Y)$ is defined to be the quotient space, which is path connected.
When $Y$ is a single point, we denote it by $BX$ for short.

We will list some properties on the space from \cite{FRS}.
By observing the cellular complexes, the following isomorphisms are known:
\begin{equation}
\label{113} H_n^R(X) \cong H_n(BX;\Z) , \ \ \ \ \ \ \ \ \ H_{n+1}^R(X) \cong \bigoplus_{i\in O_X} H_{n}^R (B(X,X_i);\Z).
\end{equation}
Furthermore, concerning fundamental groups, we mention the following isomorphisms \cite{FRS}:
\begin{equation}
\label{114}\pi_1(BX)\cong \As(X), \ \ \ \ \ \ \ \ \ \pi_1(B(X,X_i)) \cong \mathrm{Stab}(x_i) \subset \As(X). \end{equation}
It is shown \cite[Proposition 5.2]{FRS} that the action of $\pi_1(BX)$ on $\pi_*(BX)$ is trivial, and the projection $ p: B(X,X_i) \ra BX$ is a covering.
Therefore, we have functorially the Postnikov tower written in
\begin{equation}
\label{115}{\normalsize
\xymatrix{
H_3(\mathrm{Stab}(x_i) ) \ar[r] \ar[d]_{ p_* }& \pi_2(B(X,X_i) ) \ar[d]_{ p_* }^{\cong } \ar[r] & H_2(B(X,X_i) ) \ar[d]_{p_* } \ar[r] & H_2^{\rm gr}(\mathrm{Stab}(x_i) ) \ar[d]_{ p_* } \ar[r] & 0 & (\mathrm{exact})
\\
H_3(\pi_1(BX) ) \ar[r] & \pi_2(BX ) \ar[r] & H_2(BX) \ar[r] & H_2^{\rm gr}( \pi_1(BX) ) \ar[r] & 0 & (\mathrm{exact}).
}}\end{equation}
Furthermore, we mention \cite[Theorem 7]{LN}, which claims the isomorphisms
\begin{equation}\label{kiddd2n}H_2^R (X) \cong H_2^Q(X) \oplus \Z^{\oplus O(X)}, \ \ \ \ \ \ \ \ H_3^R (X) \cong H_3^Q( X ) \oplus H_2^Q(X) \oplus \Z^{\oplus O(X) \times O(X)}.
\end{equation}

Using the above results, we will show the following theorem:

\begin{thm}\label{3331}
Let $Q_L$ be the link quandle of a non-trivial knot $L$. Then,
$$ H_3^R( Q_L) \cong \Z \oplus \Z \oplus \Z, \ \ \ \ \ \ H_3^Q( Q_L) \cong \Z.$$
Furthermore, the quandle homology $H_3^Q( Q_L) \cong \Z$ is generated by the fundamental 3-class [$\mathcal{S}_{\mathrm{id}_{Q_L}}$] (recall Example \ref{r11} for the definition).
\end{thm}
\begin{proof}
Note $O(X)=1$. Then, $ \pi_1(B(X,X)) \cong \mathrm{Stab}(x_0)$ is the peripheral group $ \cong \Z^2$.
Hence, $H_*^{\rm gr}(\mathrm{Stab}(x_i) ) \cong H_*( S^1 \times S^1;\Z )$.
Furthermore, $ \pi_2(BX ) \cong \Z^2$ is known \cite{FRS}.
Hence, the above sequence deduces $H_3^R( Q_L) \cong \Z^3$ as desired.
Furthermore, since $H_2^R( Q_L)\cong H_1(B(Q_L,Q_L))\cong H_1^{\rm gr}(\mathrm{Stab}(x_i) ) \cong \Z ^2$ (which is already known \cite{RS,FRS}), we have $H_2^Q( Q_L) \cong \Z $ by \eqref{114}.
Hence, we obviously obtain $H_3^Q( Q_L) \cong \Z$ from \eqref{kiddd2n}.

We complete the remaining proof on the generator.
Consider a chain map $ q: C_n^R(X) \ra C_{n-1}^R(X) $ induced by $ (x_1,\dots, x_n) \mapsto (x_2,\dots, x_n )$.
We can easily check that the map on the homology level coincides with the above $p_* : H_n(BX) \ra H_{n-1}(BX) $.
Here note the fact \cite{FRS,RS} that $H_2^Q(Q_L) \cong \Z $ is generated by the 2-class $q_*[\mathcal{S}_{\mathrm{id}_{Q_L}}] $.
Hence, $H_3^Q( Q_L) $ must be generated by the 3-class [$\mathcal{S}_{\mathrm{id}_{Q_L}}$].
\end{proof}
Next, we will deal with the link case. As seen in Remark \ref{ssss} or Seifert pieces, it is complicated to deal with $H_3^Q(Q_L)$ of every links.
Thus, we assume a property:
$$ \textrm{ The centralizer subgroup of each pherihedra group }\mathcal{P}_i \subset \pi_L \textrm{ is equal to } \mathcal{P}_i.\ \ \ \ \ \ \ \ \ (\dagger).$$
For example, if $L$ is hyperbolic or if the subgroups $ \mathcal{P}_i $ are malnormal, the property holds.
\begin{thm}\label{333221}
Let $L \subset S^3 $ be a non-splitting link with the property ($\dagger$). Then,
$$ H_3^R( Q_L) \cong \Z^{\# L(\# L +3 )}, \ \ \ H_3^Q( Q_L) \cong \Z^{\# L}.$$
\end{thm}

\begin{proof}
First, we show $ H_2^R( Q_L ) \cong \Z^{ 2 \# L }$.
By ($\dagger$), we can easily verify $ \mathrm{Stab}(x_i) \cong \Z^2 $ and $|O(Q_L)|=\# L$.
Hence, it follows from \eqref{113} that $H_2^R( Q_L ) \cong \oplus_{i \in O(X)} H_1^{\rm gr}(\mathrm{Stab}(x_i) ) \cong \Z^{ 2 \# L}$ as desired.

Next, we show $ \pi_2(BX ) \cong \Z^{ \# L +1} $.
Since $L$ is non-splitting, $S^3 \setminus L$ is an Eilenberg-MacLane space of $\pi_L $; see, e.g., \cite{AFW}.
Accordingly, $ H_*^{\rm gr}(\As(Q_L)) \cong H_*(S^3 \setminus L) $.
Thus, $ H_3^{\rm gr}(\As(Q_L)) \cong 0$ and $ H_2^{\rm gr}(\As(Q_L)) \cong \Z^{\# L -1 }$.
By \eqref{115} and $ H_2(B Q_L ) \cong \Z^{ 2 \# L }$, we obviously have $ \pi_2(BQ_L ) \cong \Z^{ \# L +1}$.

Next, we will complete the proof.
Since $ H_2(B Q_L ) \cong \Z^{ 2 \# L}$ and $ \mathrm{Stab}(x_i) \cong \Z^2 $, we have $H_2(B (X,X_i) ) \cong \Z^3$ by \eqref{115}.
Hence, we have $H_3^R( Q_L) \cong \Z^{\# L(\# L +3 )}$ by \eqref{113} with $n=2$.
Furthermore, regarding $H_3^Q( Q_L)$, the proof is readily due to \eqref{kiddd2n}.
\end{proof}

Finally, we prove Lemma \ref{4321} which is used in \S \ref{s33}.
In what follows, we use the notation in \S \ref{s33}. More precisely, we should recall the associated quandles $Q_L$ and $ Q_{L'}$ of links $L$ and $L'$, respectively,
and employ the homogenous quandle of the form $ Q_W =( K_1 \backslash \pi_{L})\sqcup (K_2 \backslash \pi_{L}) $,
together with the injections $i_*: Q_L \ra Q_W$ and $j_*: Q_{L'} \ra Q_W$; see \S \ref{s33} for the details. Furthermore, we assume $L \subset L' $ and the hyperbolicity of $S^3 \setminus L \setminus ( S^3 \setminus L'). $ 

\begin{lem}\label{4321} Assume that $L'$ is a hyperbolic link.
Then, $H_3^Q( Q_W) \cong \Z^2$. Furthermore, the induced maps on the third homology level
$$i_*^Q: H_3^Q(Q_L) \lra H_3^Q(Q_W), \ \ \ \mathrm{and} \ \ \ j_*^Q: H_3^Q (Q_{L'}) \lra H_3^Q(Q_W)$$
are given by $(1,0)$ and $ {\small \left(
\begin{array}{cc}
1 & 1 \\
1 & 1 \end{array}
\right) }$, respectively.
\end{lem}
\begin{proof} 
By the form of $Q_W$, we can easily show $\As(Q_W) \cong \pi_1(S^3 \setminus L)$.
Further, notice the amalgamation
$$\pi_L = \pi_1(S^3 \setminus V_1) *_{\pi_1( \partial (S^3 \setminus V_1) )} \pi_1(V_1\setminus L).$$
Since $V_1\setminus L$ is hyperbolic, and $S^3 \setminus V_1$ is a knot,
the centralizer of $\pi_1( \partial (S^3 \setminus V_1) )$ in $\pi_L$ is itself isomorphic to $\Z^2.$ Thus, $ \mathrm{Stab}(x_i) \cong \Z^2 $ and $|O(X)|=2 $.
Therefore, the isomorphism \eqref{113} readily implies
$ H_2^R( Q_W ) \cong \Z^{ 4 }.$
Then, by the same discussion of Theorem \ref{333221}, we can show $H_3^Q( Q_W) \cong \Z^2$ .

Furthermore, by the above proofs, we already know the basis of second quandle homologies, as the stabilizer subgroups. Thus, by functoriality from boundary inclusions, 
we can check the matrix presentations. 
\end{proof}

\vskip 1pc

\normalsize

\normalsize
DEPARTMENT OF MATHEMATICS TOKYO INSTITUTE OF TECHNOLOGY
2-12-1
OOKAYAMA
, MEGURO-KU TOKYO
152-8551 JAPAN

\

E-mail address: {\tt nosaka@math.titech.ac.jp}

\end{document}